\documentclass[a4paper,11pt]{amsart}
\usepackage[foot]{amsaddr}
\usepackage{geometry}
\usepackage{amsfonts}
\usepackage{amsthm}
\usepackage{amssymb}
\usepackage{amsmath}
\usepackage{theoremref}
\usepackage{enumerate}
\usepackage{bbm}
\usepackage{bm}
\usepackage{mathtools}
\usepackage{a4wide}
\usepackage{xcolor}
\usepackage{lipsum}
\usepackage{tikz-cd}
\usepackage{MnSymbol}
\usepackage{dsfont} 

\makeatletter
\newcommand{\proofpart}[2]{%
  \par
  \addvspace{\medskipamount}%
  \noindent\emph{Step #1: #2}\par\nobreak
  \addvspace{\smallskipamount}%
  \@afterheading
}
\makeatother

\DeclarePairedDelimiter\abs{\lvert}{\rvert}%
\DeclarePairedDelimiter\norm{\lVert}{\rVert}%

\makeatletter
\let\oldabs\abs
\def\abs{\@ifstar{\oldabs}{\oldabs*}}
\let\oldnorm\norm
\def\norm{\@ifstar{\oldnorm}{\oldnorm*}}
\makeatother

\makeatletter
\g@addto@macro\bfseries{\boldmath}
\makeatother

\newcommand{\A}{\mathcal{A}}

\newcommand{\Dy}{\mathcal{D}}

\newcommand{\N}{\mathcal{N}}
\newcommand{\T}{\partial\mathbb{D}}
\newcommand{\G}{\mathcal{G}}

\newcommand{\z}{\zeta}

\newcommand{\conj}[1]{\overline{#1}}
\newcommand{\D}{\mathbb{D}}

\newcommand{\cD}{\conj{\mathbb{D}}}

\renewcommand{\Dy}{\mathcal{D}}

\renewcommand\Re{\operatorname{Re}}

\newtheorem{thm}{Theorem}[section]
\newtheorem{lemma}[thm]{Lemma}

\newtheorem{cor}[thm]{Corollary}
\newtheorem{prop}[thm]{Proposition}
\theoremstyle{definition}

\newtheorem{definition}{Definition}[section]

\begin{document}
\title{\textbf{Shift invariant subspaces in growth spaces and sets of finite entropy}}

\author{Adem Limani} 
\address{Departament de Matem\`atiques, Universitat Autònoma de Barcelona, Barcelona 08193}
\email{AdemLimani@mat.uab.cat}
\date{\today}

\maketitle

\begin{abstract} \noindent
We investigate certain classes of shift invariant subspaces in growth spaces on the unit disc of the complex plane determined by a majorant $w$, which include the classical Korenblum growth spaces. Our main result provides a complete description of shift invariant subspaces generated by Nevanlinna class functions in growth spaces, where we show that they are of Beurling-type. In particular, our result generalizes the celebrated Korenblum-Roberts Theorem. It turns out that singular inner functions play the decisive role in our description, phrased in terms of certain $w$-entropy conditions on the carrier sets of the associated singular measures, which arise in connection to boundary zero sets for analytic functions in the unit disc having modulus of continuity not exceeding $w$ on the unit circle. Furthermore, this enables us to establish an intimate link between shift invariant subspace generated by inner functions and the containment of the above mentioned analytic function spaces in the corresponding model spaces.

\end{abstract}

\section{Introduction}\label{Sec:Intro} 
\subsection{Basic notions}
Let $w$ be a weight, that is, a continuous non-decreasing function defined on the unit-interval $[0,1]$ with $w(0)=0$ and denote by $\G_w$ the corresponding growth space of analytic functions $g$ in the unit disc $\D$ satisfying
\[
\lim_{|z| \to 1-} w(1-|z|) g(z) = 0.
\]
The space $\G_w$ is obtained by taking the closure of analytic polynomials in the norm 
\[
\norm{g}_{\G_w} := \sup_{z \in \D } w(1-|z|)|g(z)| < \infty,
\]
hence equipped with the above norm, the set $\G_w$ becomes a separable Banach space for which analytic polynomials form a dense linear span. Typical examples of growth spaces are the classical Korenblum growth spaces, which correspond to weights of the form $w(t) = t^{\alpha}$ with $\alpha >0$. For an extensive treatment of these spaces, we refer the reader to the survey in \cite{hedenmalmbergmanspaces}. We denote by $H^\infty$ the algebra of bounded analytic functions in $\D$ equipped with the supremum norm $\norm{f}_{H^\infty}= \sup\left\{\abs{f(z)}: z\in \D \right\}$. It is not difficult to verify that $H^\infty$ is the multiplier algebra of $\G_w$, that is, $fg \in \G_w$ for any $g\in \G_w$ precisely when $f \in H^\infty$. Let $\mathcal{N}$ denote the Nevanlinna class of analytic functions $f$ in $\D$ satisfying
\[
\sup_{0<r<1} \int_{\T} \log^+ \abs{f(r\zeta)} dm(\zeta) < +\infty,
\]
where $dm$ denotes the unit-normalized Lebesgue measure on the unit-circle $\T$. It is a standard fact that any function $f$ belonging to $\mathcal{N}$ enjoys the unique Herglotz-Nevanlinna factorization $f= \mathcal{O}_f (S_{\mu_1}/S_{\mu_2})B_{\Lambda} $, where $\mathcal{O}_f$ is the \emph{outer} factor of $f$ defined by 
\[
\mathcal{O}_f (z) = \exp \left( \int_{\T} \frac{\zeta+z}{\zeta-z} \log \abs{f(\zeta)} dm(\zeta) \right), \qquad z\in \D
\]
where $\mu_k$, $k=1,2$ are mutually singular positive finite Borel measures on $\T$, both singular wrt $dm$ and induce corresponding \emph{singular inner functions} given by
\[
S_{\mu_k}(z) = \exp \left( - \int_{\T} \frac{\zeta+z}{\zeta-z} d\mu_k(\zeta) \right) \qquad z\in \D,\qquad k=1,2,
\]
and $B_{\Lambda}$ denotes the Blaschke product 
\[
B_{\Lambda}(z) = e^{i\gamma} \prod_{\lambda\in \Lambda} \frac{\abs{\lambda}}{\lambda}\frac{\lambda-z}{1-\conj{\lambda}z}, \qquad z\in \D,
\]
corresponding to the \emph{Blaschke sequence} $\Lambda \subset \D$ which  encodes the zeros of $f$ in $\D$ satisfying
\[
\sum_{\lambda \in \Lambda} (1-|\lambda|) < \infty,
\]
and $\gamma$ is a real-valued constant. A bounded analytic function $\Theta$ in $\D$ with unimodular radial limits $\lim_{r\to 1-} \abs{\Theta(r\zeta)}=1$ at $dm$-a.e every point $\zeta \in \T$ is said to be an \emph{inner function} and all inner functions are of the form $\Theta=B_{\Lambda}S_{\mu}$ for some Blaschke sequence $\Lambda$ and positive finite measure $\mu$ singular wrt to $dm$ on $\T$. For sake of abbreviation, we reserve the above notations the Herglotz-Nevanlinna factors in $\mathcal{N}$ and their corresponding factors, and we refer the reader to the survey in \cite{cauchytransform} for further information on this theory. Given a function $g$ analytic in $\D$ we denote by $M_z$ the \emph{shift operator} $M_z(g)(z) = zg(z)$, which corresponds to shifting the Taylor coefficients of $g$ forward by one unit. A closed linear subspace $\mathcal{V}$ of a Banach space $X$ of analytic functions in $\D$ is said to be shift invariant (or $M_z$-invariant), if  $M_z \mathcal{V} \subseteq \mathcal{V}$. Let $H^2 \subset \mathcal{N}$ denote the Hardy spaces consisting of functions
\[
\norm{f}^2_{H^2} := \sup_{0<r<1} \int_{\T} \abs{f(r\zeta)}^2 dm(\zeta) < \infty.
\]
\subsection{Previous work}
The celebrated Beurling Theorem on the function theoretical description of shift invariant subspaces in $H^2$ asserts that any shift invariant subspace $\mathcal{V} \subseteq H^2$ is for the form 
\[
\mathcal{V} = \Theta_{\mathcal{V}} H^2 := \left\{ \Theta_{\mathcal{V}} f: f\in H^2 \right\}
\]
for a uniquely associated inner function $\Theta_{\mathcal{V}}$. Beurling's elegant description of shift invariant subspaces has been an exemplary guide towards investigating the structures of $M_z$-invariant subspaces on various other analytic function spaces, such as the Bergman spaces, Korenblum growth spaces and the Dirichlet spaces, for instance see \cite{korenblum1977beurling}, \cite{aleman1996beurling}, \cite{richter1992multipliers} and further sources therein. In the context of growth spaces, it is known that the structure of $M_z$-invariant subspaces in $\G_w$ is very complicated, as they can be of arbitrary large index, see \cite{borichev1998invariant}. 

Our principal goal here is to describe an important class of singly generated $M_z$-invariant subspaces in $\G_w$. To this end, we shall denote by $\left[ g \right]_{\G_w}$ the smallest closed $M_z$-invariant subspace containing the element $g$ in $\G_w$, which we refer to as the $M_z$-invariant subspace generated by $g$. A function $g$ is declared to be cyclic in $\G_w$ if $\left[ g \right]_{\G_w} = \G_w$. In other words, $g$ is cyclic in $\G_w$ if there exists a sequence of analytic polynomials $\{Q_n\}_n$ such that 
\[
\lim_n \norm{gQ_n -1}_{\G_w} = 0.
\]
Indeed, since the analytic polynomials are dense in $\G_w$, we have that $g$ is cyclic in $\G_w$ if and only if $1\in \left[ g\right]_{\G_w}$.  An often useful observation is that $\left[g\right]_{\G_w}$ actually equals the closure of $g H^\infty$ in the norm of $\G_w$, hence we may substitute analytic polynomials $\{Q_n\}_n$ as above by a sequence of bounded analytic functions. A complete description of shift invariant subspaces in Bergman spaces generated by functions in the Nevanlinna class was obtained in the work of H. Hedenmalm, B. Korenblum and K. Zhu in \cite{hedenmalm1996beurling}, see Theorem A. In the context of Korenblum growth spaces, their result can be phrased in the following way. Let $w(t)=t^\alpha$ with $\alpha>0$ be a majorant and $f\in \G_w \cap \mathcal{N}$ with Herglotz-Nevanlinna factorization $f= \mathcal{O}_f (S_\mu / S_\nu) B_{\Lambda}$. Then the $M_z$-invariant subspace in $\G_w$ generated by $f$ is of \emph{Beurling-type}, that is, it is uniquely determined by the Blaschke sequence $\Lambda$ encoding the zeros of $f$ in $\D$ and the positive singular Borel measure $\mu_f \leq \mu$, which is the part of $\mu$ carried on a countable union of Beurling-Carleson sets, that is, countable unions of closed sets $E\subset \T$ of Lebesgue measure zero satisfying
\[
\sum_k m(I_k) \log m(I_k) >- \infty
\]
where $\{I_k\}_k$ are the connected components of $\T \setminus E$. More precisely, the $M_z$-invariant subspace generated by $f$ equals $\left[ f\right]_{\G_w} = \left[B_{\Lambda} S_{\mu_f}\right]_{\G_w}$
and the shift invariant subspace generated by $BS_{\mu_f}$ satisfies the following characterizing \emph{permanence property}:
\begin{equation}\label{P-prop}
\left[f\right]_{\G_w}  \cap \mathcal{N}_+ \subseteq B_{\Lambda} S_{\mu_f} \mathcal{N}_+.
\end{equation}
Here $\mathcal{N}_+$ denotes the Smirnov class of functions in $\mathcal{N}$ with (unique) Herglotz-Nevanlinna of the form $f= \mathcal{O}_f S_\mu B_{\Lambda}$. In fact, their results is phrased in terms of a generalized Nevanlinna representation of functions in Korenblum growth spaces using notions of so-called pre-measures, initially introduced in the pioneering works of B. Korenblum in \cite{korenblum1975extension}, \cite{korenblum1977beurling}.
\subsection{Main results}

In this note, our purpose is to extend the aforementioned result to growth spaces $\G_w$ associated to weights $w$ of slower than polynomial decay. A continuous non-decreasing and sub-additive function $w$ on $[0,1]$ with $w(0)=0$ is said to be a \emph{modulus of continuity}. Given a positive finite continuous (no atoms) Borel measure $\nu$ on $\T$ ones defines
\[
\omega_{\nu}(\delta) := \left\{\nu(I): m(I)\leq \delta, \,\, \text{arcs}\, \, I\subseteq \T \right\},
\]
which is easily verified to be a modulus of continuity, usually referred to as \emph{the modulus of continuity of $\nu$}. Conversely, one can also show that any modulus of continuity arises as $\omega_{\nu}$ for some positive finite continuous Borel measure $\nu$ (for instance, see \cite{berman1987moduli}). Note that if $w$ is a modulus of continuity, then $w^{\lambda}$ need only to be modulus of continuity when $0<\lambda\leq1$, hence in order to compensate for this lack of stability we make the following definition. 

\begin{definition}
A continuous non-decreasing function $w$ on the unit-interval $[0,1]$ with $w(0)=0$ is said to be a \emph{majorant} if there exists a number $\lambda= \lambda(w)>0$ such that $w^\lambda$ is a modulus of continuity. 
\end{definition}
As a consequence, $w$ is a majorant if and only if $w^\lambda$ is a majorant for any $\lambda>0$. We shall also need to introduce a notion which generalizes classical Beurling-Carleson sets. Given a majorant $w$, we declare a closed set $E \subset \T$ of Lebesgue measure zero to have finite $w$-entropy if
\begin{equation}\label{wset}
 \int_{\T} \log w \left( \textbf{dist}(\zeta,E) \right) dm(\zeta) >- \infty
\end{equation}
where $\textbf{dist}(\zeta,E)$ denotes the unit-normalized distance from $\zeta$ to $E$ along $\T$. An important observation is that any closed subset of finite $w$-entropy again has finite $w$-entropy, and that sets of finite $w$-entropy are closed under the formation of finite unions. Indeed, the later follows from the observation
\[
\textbf{dist}(\zeta,\cup_{k=1}^n E_k ) = \min_{1\leq k\leq n} \textbf{dist}(\zeta, E_k ), \qquad \zeta \in \T.
\]
Given a majorant $w$ and a positive finite singular Borel measure $\mu$ on $\T$, we consider the quantity
\[
\sup \left\{\mu(E): E \subset \T \, \, \text{of finite}  \, w \text{-entropy} \right\}.
\]
By definition, we can pick a sequence of sets $\{E_n\}_n$ where each $E_n$ has finite $w$-entropy, which we may with out loss of generality assume to be increasing, such that $\mu(E_n)$ converges to the supremum above. It is straightforward to verify that $\mu$ decomposes uniquely (up to $\mu$-null sets) according to
\begin{equation}\label{PCdecomp}
    \mu = \mu_P + \mu_C
\end{equation}
where $\mu_P$ is concentrated on $\cup_n E_n$, a countable union of sets having finite $w$-entropy and $\mu_C$ does not charge any set of finite $w$-entropy.
Proceeding forward, we shall make the following assumption on our majorants $w$:
 \begin{enumerate}\label{A_1cond}
   
 \item[$(A_1)$] 
    \begin{equation*} 
\log \frac{1}{w(t)} \asymp \log \frac{1}{w(t^2)} , \qquad 0 < t < 1.
  \end{equation*}

\end{enumerate}
Here by $A \asymp B$ we mean that there exists a constant $c>0$ such that $c^{-1}B \leq A \leq c B$ holds. Note that the above condition essentially ensures that the majorants $w$ do not decay to zero too fast, and examples of eligible majorants satisfying the above condition are given by $t^c$, $\log^{-c}(e/t)$, $(\log \abs{\log (e / t)})^{-c}$, $(\log \abs{\log \abs{\log (e/t)}})^{-c}, \dots$, $c >0$ and arbitrary many logarithmic compositions of such kind. Hence condition $(A_1)$ essentially puts no restriction on how slow $w$ tends to zero. We wish to point out that for the sake of clarity in exposition, we decided to solely restrict our attention to majorants $w$, but with certain modifications our results continue to hold for weights which decay as fast as $w(t)= \exp(-\alpha \log^{\beta}(e/t) )$ with $\alpha, \beta>0$, as they also satisfy condition  $(A_1)$. Further comments on these matters will be discussed in Section 4. We now phrase our main result, which gives a complete description of the shift invariant subspaces in $\G_w$ generated by functions in the Nevanlinna class $\mathcal{N}$.




\begin{thm}\thlabel{GwNevMz}
Let $w$ be a majorant satisfying $(A_1)$ and let $f$ be a function in $\G_w \cap \mathcal{N}$ with Herglotz-Nevanlinna factorization $f=\mathcal{O}_f (S_\mu /S_\nu)B_{\Lambda}$ and let $\mu$ be decomposed according to \eqref{PCdecomp}. Then the following statements hold:
\begin{enumerate}
    
    \item[(i)] The shift invariant subspace generated by $f$ in $\G_w$ is completely determined by the corresponding pair $(\Lambda, \mu_P)$. More precisely, we have $\left[f\right]_{\G_w}= \left[ B_{\Lambda} S_{\mu_P} \right]_{\G_w}$ where the following permanence property holds:
    \[
    \left[ f \right]_{\G_w} \cap \mathcal{N}_+ \subseteq B_{\Lambda} S_{\mu_P} \mathcal{N}_+.
    \]
    \item[(ii)] The singular inner function $S_{\nu}$ is cyclic in $\G_w$.
    
\end{enumerate}
 In particular, $f$ is cyclic in $\G_w$ if and only if it is of the form $f= \mathcal{O}_f S_{\mu_C}/S_{\nu}$.

\end{thm}
\noindent 
We remark that part $(i)$ asserts that any $M_z$-invariant subspace in $\G_w$ generated by a function $f$ in $\mathcal{N} \cap \G_w$ with Nevanlinna factorization according to the statement above, is completely determined by the corresponding pair $(\Lambda, \mu_P)$, which consists of the corresponding Blaschke sequence $\Lambda$ and the associated singular measure $\mu_P$, which is the "maximal" part of $\mu$ concentrated on a countable union of sets having finite $w$-entropy. In the language of \cite{korenblum1977beurling}, such subspaces are said to be of \emph{Beurling-type}, hence our result extends Theorem A therein to the context of growth spaces $\G_w$ for majorants $w$. Admittedly, the characterization in \cite{korenblum1977beurling} is phrased in terms of a generalized Nevanlinna factorization of Korenblum growth spaces using the notion of \emph{pre-measures}, and a similar description can be carried out in the context of $\G_w$ for certain majorants $w$ of Shannon-type (see Theorem 1, p.543 in \cite{korenblum1985class} and also \cite{hanine2015cyclic} for elaborate details), using similar methods of proofs as in \cite{hedenmalm1996beurling}. Our result also asserts that the singular inner factor $S_\nu$ appearing in denominator of an element in $\G_w \cap \mathcal{N}$ must necessarily be cyclic in $\G_w$. This observation and statement $(ii)$ have previously appeared in the contexts of Bergman spaces in \cite{berman1984cyclic} and \cite{bourdon1985cyclic}. As a corollary of \thref{GwNevMz}, we obtain the following generalization of the Korenblum-Roberts Theorem in the context of $\G_w$. 

\begin{cor} \thlabel{Thm:GenKR} Let $w$ be a majorant satisfying $(A_1)$ and suppose that $\mu$ is a positive finite Borel singular measure on $\T$, decomposed according to \eqref{PCdecomp}. Then for any $H^\infty$-function $f$ with Herglotz-Nevanlinna factorization $f=\mathcal{O}_f S_{\mu} B_{\Lambda}$ the following dichotomy holds:
\begin{enumerate}

\item[(i)] The inner function $\Theta_0 := B_{\Lambda}S_{\mu_P}$ generates a proper $M_z$-invariant subspace of $\G_w$ satisfying the permanence property:
\[
\left[\Theta_0\right]_{\G_w} \cap \mathcal{N}_+ \subseteq \Theta_0 \, \mathcal{N}_+.
\]

\item[(ii)] $\mathcal{O}_f S_{\mu_C}$ is cyclic in $\G_w$.

\end{enumerate}
Moreover, the shift invariant subspace generated by $f$ is completely determined by the corresponding pair $(\Lambda, \mu_P)$ and equals $\left[B_{\Lambda} S_{\mu_P}\right]_{\G_w}$.
\end{cor}
\noindent
For weights $w(t)=t^\alpha$ with $\alpha>0$, one retains the celebrated Korenblum-Roberts Theorem, see Theorem 2 in \cite{roberts1985cyclic} and the main result in \cite{korenblum1981cyclic}. We remark that our result on cyclicity in $(ii)$ actually holds for any majorant $w$ and does not require the presence of $(A_1)$. The existence of a cyclic singular inner function in any growth space $\G_w$ was established a long time ago by H. Shapiro in \cite{shapiro1967some} (see also \cite{malman2022cyclic} for recent further work in this direction), hence we complete this picture by classifying all cyclic singular inner functions in $\G_w$. In fact, the following observation is a consequence of \thref{Thm:GenKR}: a function $f \in H^\infty$ is cyclic in $\G_w$ if and only if $f$ has no zeros in $\D$ and its associated singular inner factor does not charge any set of finite $w$-entropy. Note that since $E$ is a set of finite $w$-entropy if and only if $E$ is a set of finite $w^{\lambda}$-entropy for any $\lambda>0$, our result on cyclicity also carries over to the setting of spaces of functions $f$ analytic in $\D$ and furnished with a Bergman-type metric of the form
\[
\left(\int_{\D} \abs{f(z)}^p w(1-|z|) dA(z) \right)^{\min(1,1/p)}
\]
for any $0<p<\infty$. Furthermore, it can also be phrased in the setting of
\[
\G_w^{\infty} := \bigcup_{\lambda>0} \G_{w^\lambda}
\]
equipped with the locally convex inductive-limit topology, that is, the largest locally convex topology for which each growth space $\G_{w^\lambda}$, $\lambda>0$ is continuously embedded into $\G_w^{\infty}$. We mention that a version of our result on cyclicity of inner functions have previously appeared (see \cite{hanine2015cyclic}) under a couple of different and more restrictive regularity assumptions on $w$, using different methods which originate back to the work of B. Korenblum in \cite{korenblum1975extension}, which in fact enables the author to describe cyclic vectors beyond inner functions. Furthermore, cyclicity of singular inner functions on certain weighted Bergman spaces defined in terms of growth restrictions on Taylor coefficients, has also previously been treated in the work of \cite{bourhim2004boundary}, where similar notions of entropy play the decisive role. The novelty here is that we only require a very mild assumption on $w$, which is enough to cover majorants $w$ of arbitrary slow rate of decay. In the context of majorants decaying at a certain exponential rate, one encounters completely different phenomenons, as notions of entropy-type loose their principal role in determining cyclicity of singular inner functions in the corresponding growth spaces. For further details on this matter, we refer the reader to recent works in \cite{el2012cyclicity}, \cite{borichev2014cyclicity} and references therein.





 
\subsection{Application to Model spaces}
Proceeding forward, we shall now apply our descriptions on shift invariant subspaces generated by inner functions $\Theta$ in $\G_w$ to demonstrate an intimate connection to the membership of functions in the associated model spaces $K_\Theta$ enjoying certain $w$-smoothness properties on $\T$. Given an inner function $\Theta$, we recall that the corresponding model space $K_\Theta$ is defined as the orthogonal complement of $\Theta H^2 =\{\Theta h: h\in H^2\}$ in the Hardy space $H^2$, that is
\[
K_{\Theta} := H^2 \ominus \Theta H^2.
\]
As a consequence of Beurling's Theorem, the model spaces are the only closed invariant subspaces for adjoint shift $M_z^*$ on $H^2$ and play a fundamental role in complex function theory and operator theory, see \cite{cauchytransform}. Given a \emph{modulus of continuity} $w$, we define by $\A_w$ the Banach space of bounded analytic functions $f$ on $\D$ with continuous extension to $\T$ and whose modulus of continuity is controlled by $w$ on $\T$, equipped with the norm
\begin{equation}\label{DefAw}
\norm{f}_{\A_w}:= \norm{f}_{H^\infty} + \sup_{\xi,\zeta \in \T} \frac{\abs{f(\xi)-f(\zeta)}}{w(\abs{\xi - \zeta})}.
\end{equation}
In other words, $\A_w:= H^\infty \cap C_w(\T)$, where $C_w(\T)$ denotes the space of continuous functions on $\T$ having modulus of continuity no worse than $w$. It follows from the work in \cite{tamrazov1973contour} that functions in $\A_w$ actually have modulus of continuity not exceeding $w$ in all of $\cD$ (see also \cite{bouya2008closed} for a short proof). In a similar way as the classical Beurling-Carleson sets are boundary zero sets of analytic functions in $\D$ with smooth extensions to $\T$ (see \cite{carlesonuniqueness} and \cite{taylor1970ideals}), it turns out that sets of finite $w$-entropy correspond to boundary zero sets of bounded analytic functions in $\D$ with continuous extension to $\T$, having modulus of continuity not exceeding $w$ on $\T$, see \cite{shirokov1982zero} and the definition of $\A_w$ in \eqref{DefAw} below. It was recently established by B. Malman in \cite{malman2022cyclic} that for any modulus of continuity $w$, there exists a singular inner function $\Theta$ such that $\A_w \cap K_\Theta = \{0\}$. The mentioned result should be put in contrast to the celebrated density theorem of A.B. Aleksandrov in \cite{aleksandrovinv} (see also \cite{limani2022abstract} for a slight improvement), which asserts that functions extending continuously to $\T$ in any model space $K_\Theta$ are always dense in $K_\Theta$. In the case when $w(t)=t^\alpha$ with $0<\alpha < 1$ is H\"olderian, K. Dyakonov and D. Khavinson proved that the intersection $\A_w \cap K_\Theta$ is non-trivial if and only if either $\Theta$ has a zero in $\D$ or the associated singular measure charges a Beurling-Carleson set on $\T$, see \cite{dyakonov2006smooth}. More recently, the author in collaboration with B. Malman showed that $\A_w \cap K_\Theta$ is dense in $K_\Theta$ if and only if the corresponding singular measure associated to $\Theta$ is concentrated on a countable union of Beurling-Carleson sets, see \cite{limani2022model}.

To this end, we make the following assumption on the modulus of continuity $w$: there exists a constant $0<\alpha < 1$, such that $w^{1+\alpha}$ is a modulus of continuity and
\begin{enumerate}\label{Bcond}
    \item[$(A_2)$]
    \begin{equation*} 
 \int_{0}^1 w^{\alpha}(t) \frac{dt}{t} < \infty.
\end{equation*}

\end{enumerate}
In other words, we essentially ask for the modulus of continuity $w$ to be just slightly better than \emph{Dini-continuous}. Besides the class of H\"olderian weights, this also includes weights of logarithmic decay with typical examples provided by $w(t) = \log^{-c}(e/t)$, for any $c>1$. Under this additional assumption we retrieve the following complete structure theorem for model spaces, which generalizes the main result in \cite{limani2022model}.



\begin{thm}\thlabel{THM:Model} 
Let $w$ be a modulus of continuity satisfying condition $(A_2)$ and let $\Theta= B S_\mu$ be an inner function where $B$ denotes its Blaschke factor. Then the following statements hold: 
\begin{description}

\item[\text{Existence}] 
$\A_w \cap K_{\Theta}= \{0\}$ if and only if $\Theta=S_{\mu_C}$. 
\item[\text{Density}] 
$\A_w \cap K_{\Theta}$ is dense in $K_\Theta$ if and only if $\Theta= B S_{\mu_P}$.

\end{description}
Furthermore, the closure of $\A_w \cap K_{\Theta}$ in $K_{\Theta}$ equals $K_{\Theta_P}$, where $\Theta_P = BS_{\mu_P}$.
\end{thm}
The above result in conjunction with \thref{Thm:GenKR} confirms an intimate connection between shift invariant subspaces generated by inner functions $\Theta$ in $\G_w$, the containment of $\A_w$-functions in the model spaces $K_\Theta$, and to boundary zero sets of $\A_w$-functions described in terms of $w$-entropy. We remark that for a modulus of continuity $w$ which lacks additional regularity conditions of Dini-type such as $(A_2)$, one encounters several principal difficulties. First, the Cauchy projection of functions in $C_{w}(\T)$ need not to be continuous on $\T$,  and secondly, the corresponding class $\A_w$ is no longer characterized in terms of pseudocontinuation or in terms of a growth condition on derivatives. Thirdly, the so-called "$K$-property" which asserts that the class $\A_w$ is invariant under any Toeplitz operator with bounded co-analytic symbol does not hold, which creates a fundamental barrier for constructing functions in $\A_w$ belonging to certain model spaces in the case when $w$ is a general modulus of continuity. For further discussions on these matters, we refer the reader to \cite{girela2006toeplitz} and references therein, and to the proof of \thref{P+reg} below. We remark that using Theorem 1.1 in \cite{limani2022abstract}, one can also deduce $H^p$-versions of \thref{THM:Model}. 

\subsection{Gathering of our manuscript}
This paper is organized as follows. Section $2$ is devoted to establish our result on cyclicity, which chiefly revolves around extending the techniques developed by J. Roberts in \cite{roberts1985cyclic} to our general setting, using a decomposition of singular measure and iterative approximation scheme involving the Corona Theorem in a clever manner. In section $3$, we generalize a machinery developed in the work of \cite{berman1984cyclic}, allowing us to describe shift invariant subspaces in $\G_w$ via embeddings into Hardy spaces on certain Privalov-star domains. The proofs of \thref{Thm:GenKR} and \thref{GwNevMz} will be carried out in section $4$, presented in that precise order, while the last subsection therein contains some further extensions and suggested directions of work. Section 5 is devoted to the proof of \thref{THM:Model}, where we consider the Cauchy-dual reformulation of shift invariant subspaces in $\G_w$, and utilize results on how the Cauchy projection distorts modulus of continuity in conjunction with N. A Shirokov's description on inner factors and boundary zero sets of functions in $\A_w$.

\subsection{Acknowledgements}
This work has benefited profoundly from the insightful discussions with Oleg Ivrii, Artur Nicolau and Bartosz Malman, leading to a substantial improvement of the previous version of this manuscript. 
Furthermore, I would also like to thank Omar El-Fallah for bringing to my attention a couple of important references. 

\section{Cyclic singular inner functions}
In this section, we shall derive the following result on cyclic singular inner functions in $\G_w$. 

\begin{thm}\thlabel{THM:Cycinner} Let $w$ be a majorant and $\mu$ be a positive finite singular Borel measure on $\T$. Then $S_{\mu_C}$ is cyclic in $\G_w$.
    
\end{thm}
Observe that the cyclicity theorem holds for any majorant $w$. Most of our efforts in this section will be devoted to extending the work developed by J. Roberts in \cite{roberts1985cyclic} to the general setting of growth spaces. 
 
\subsection{Describing sets of finite $w$-entropy via complementary arcs}
We shall need to rephrase the condition on finite $w$-entropy in terms of a discrete sum over complementary arcs. To this end, recall that for any majorant $w$, there exists a number $\lambda>0$ such that $w^\lambda(t)/t$ is \emph{almost-decreasing}, that is 
\begin{equation}\label{alm-dec}
    \frac{w^\lambda(t)}{t} \leq 2 \frac{w^\lambda(s)}{s}, \qquad 0<s<t<1.
\end{equation}
Indeed, pick a number $\lambda>0$ such that $w^\lambda$ is a modulus of continuity and let $m\geq 1$ be an integer with the property that $2^{m-1} \leq t/s < 2^m$ and note that by monotonicity and sub-additivity of $w^\lambda$ we deduce
\[
\frac{w^\lambda(t)}{t} \leq \frac{w^\lambda(2^{m}s)}{t} \leq  2^{m} \frac{w^\lambda(s)}{t} \leq 2\frac{w^\lambda(s)}{s}.
\]
In particular, we have $tw^\lambda(1)/2 \leq w^\lambda(t)$ and hence $\lim_{t\to 0+} t \log w(t) = 0$. The following lemma appears in \cite{el2012cyclicity} under slightly stronger assumptions on $w$, but for the readers convenience we shall include a short proof-sketch which holds for any majorant. 

\begin{lemma} \thlabel{wsetcomp}
Let $w$ be a majorant. Then a closed set $E \subset \T$ of Lebesgue measure zero has finite $w$-entropy if and only if 
\[
\sum_{k} m(I_k) \log w(m(I_k)) > -\infty.
\]
where $\{I_k\}_k$ are the connected components of $\T \setminus E$. Moreover, if $\T \setminus E = \cup_k J_k$ for some disjoint collection of open arcs $\{J_k\}_k$ satisfying 
\[
\sum_k m(J_k) \log w(J_k) >-\infty,
\]
then $E$ has finite $w$-entropy.
\end{lemma}

\begin{proof}
Using simple change of variables one shows that
\[
\int_{\T} \log w \left( \textbf{dist} ( \zeta ,E ) \right) dm(\zeta) = \sum_k \int_{0}^{m(I_k)/2} \log w(t) dt.
\]
Now since $w$ is non-decreasing, we clearly have 
\[
\int_{0}^{m(I_k)/2} \log w(t) dt \leq \frac{m(I_k)}{2} \log  w(m(I_k) ).
\]
On the other hand, let $\lambda>0$ such that $w^\lambda$ is a modulus of continuity and using the fact that $w^\lambda (t)/t$ is almost-decreasing in \eqref{alm-dec} in conjunction with integration by parts, it is not difficult to show that
\[
\int_{0}^{m(I_k)/2} \log w(t)  dt \geq m(I_k)\log  w(m(I_k) ) -\frac{(1+\log 2)}{\lambda} m(I_k),
\]
This is enough to establish the first claim. For the second assertion, we note that any open arc $J_l$ is contained in a maximal connected component $I_k$, hence since $w$ is non-decreasing we have
\[
\sum_k m(I_k) \log \frac{1}{m(I_k)} = \sum_k \sum_{l: J_l \subseteq I_k} m(J_l) \log \frac{1}{m(I_k)} \leq \sum_k \sum_{l: J_l \subseteq I_k} m(J_l) \log \frac{1}{m(J_l)} < \infty.
\]
By the previous argument, we conclude that $E$ has finite $w$-entropy.
\end{proof}

\subsection{Dyadic grids adapted to $w$}
In this section, we shall consider dyadic grids which are suitably adapted to a specific majorant $w$. Let $\Dy_n$ denote the collection of $2^n$ dyadic arcs $I$ with $m(I)=2^{-n}$. Inspired from the recent work in \cite{ivrii2022beurling}, we declare a collection of dyadic arcs $\mathcal{D}_w = \cup_{k\geq 0}\mathcal{D}_{n_k}$ to be a dyadic $w$-grid, if the sequence of positive integers $\{n_k\}_{k\geq 0}$ satisfies the following condition: there exists a positive number $\beta >0$, such that 
\begin{equation}\label{wgrid}
\sup_{k\geq 0} \frac{w^{\beta}(2^{-n_k})}{w(2^{-n_{k+1}})} < \infty.
\end{equation}
Our main result in this subsection is a lemma which allows us to exhibit for any majorant $w$, a dyadic $w$-grid with a certain special property, important for our further developments. 
\begin{lemma}\thlabel{ParameterLemma}
For any majorant $w$ and any positive integer $n_0>0$, there exists a dyadic $w$-grid $\Dy_w = \cup_{k\geq 0} \Dy_{n_k}$ with the additional property 
 \begin{equation} \label{superlacw}
 w(2^{-n_{k+1}}) \leq  \prod_{j=0}^k w(2^{-n_j}), \qquad k\geq 0.
\end{equation}
\end{lemma}

\begin{proof}
For the sake of abbreviation, we set $\eta(t) = \lambda \log 1/ w(t)$, where $\lambda>0$ such that $w^\lambda$ is modulus of continuity. We shall construct the sequence $\{n_k\}_{k\geq 0}$ recursively by means of induction. To this end, fix an arbitrary $n_0 > 0$ with $w^\lambda(2^{-n_0})<1/2$ and observe that by monotonicity and sub-additivity of $w^\lambda$, one easily deduces that
\begin{equation}\label{lambdasub}
\frac{1}{2} \eta(t/2) \leq \eta(t) \leq \eta(t/2)
\end{equation}
for any $0<t\leq 2^{-n_0}$. Fix a constant $C > 1$ to be specified later. Since $\eta$ is non-increasing and continuous, there exists $0<\delta_0 < 2^{-n_0}$ such that
\[
C \leq \frac{\eta(\delta_0)}{\eta(2^{-n_0})} < 5C.
\]
Pick $n_1 > n_0$ to be the smallest positive integer with $2^{-n_1} \leq \delta_0 < 2^{-(n_1-1)}$. It follows by the property of $\eta$ in \eqref{lambdasub} that
\[
C \leq \frac{\eta(2^{-n_1})}{\eta(2^{-n_0})} < 10C.
\]
Now by means of induction, suppose that $n_k>0$ has been chosen, and let $0< \delta_k < 2^{-n_k}$ with 
\[
C \leq \frac{\eta(\delta_k)}{\eta(2^{-n_k})} < 5C.
\]
We then choose $n_{k+1} > n_k$ be the smallest positive integer with $2^{-n_{k+1}} \leq \delta_{k} < 2^{-(n_{k+1}-1)}$. Again, using \eqref{lambdasub} it follows that 
\[
C \leq \frac{\eta(2^{-n_{k+1}})}{\eta(2^{-n_k})} < 10C.
\]
Consequently, we obtain an increasing sequence of positive integers $\{n_k\}_{k\geq0}$, which by construction satisfies \eqref{wgrid}, and thus we conclude that $\mathcal{D}_w = \cup_{k\geq 0} \Dy_{n_k}$ is dyadic $w$-grid. It remains to verify \eqref{superlacw}, which is equivalent to 
\[
\sum_{j=0}^k \eta (2^{-n_j}) \leq \eta (2^{-n_{k+1}}), \qquad k\geq 0.
\]
Observe that a straightforward iteration shows that
\[
\eta(2^{-n_j}) \leq C^{-1} \eta (2^{-n_{j+1}}) \leq \left(C^{-1}\right)^{k-j +1} \eta(2^{-n_{k+1}}), \qquad 0\leq j \leq k.
\]
Using these estimates term-wise and summing up the geometric sum, we arrive at 
\[
\sum_{j=0}^k \eta (2^{-n_j}) \leq \eta(2^{-n_{k+1}}) \sum_{j=0}^k \left(C^{-1}\right)^{k+1-j} \leq \frac{1}{C-1} \eta(2^{-n_{k+1}}).
\]
Choosing $C >2$ established the desired claim.

\end{proof}

\subsection{A structure theorem for singular measures} 
Our main objective is to extend the Roberts decomposition for singular measures with respect to sets of finite $w$-entropy. To this end, recall that the modulus of continuity of a positive finite Borel measure $\nu$ on $\T$ (not necessarily continuous) is again defined by

\[
\omega_{\nu}(\delta) := \sup \left\{ \nu(I): m(I)\leq \delta, \, \, I\subset \T \, \, \text{arc} \right\}
\]
One verifies that $\omega_{\nu}$ is non-decreasing and sub-additive on $[0,1]$ with $\omega_{\nu}(0)=0$. Moreover, the argument in \eqref{alm-dec} also shows that the function $ \omega_{\nu}(t)/t$ is almost-decreasing on $(0,1]$.
We are now ready to phrase the main result in this subsection.

\begin{prop}[Roberts decomposition] \thlabel{RobWdecomp}
Let $\mu$ be a positive finite Borel singular measure on $\T$. For any integer $n_0>0$, for any constant $c>0$ and for any dyadic $w$-grid $\Dy_w = \cup_{k\geq 0} \Dy_{n_k}$, there exists positive measures $\{\mu_k\}_{k\geq 0}, \mu_\infty$ on $\T$ such that $\mu$ decomposes as
\[
\mu = \sum_{k\geq 0} \mu_k + \mu_\infty,
\]
where each $\mu_k$ satisfies
\begin{equation}\label{modcontmuk}
\omega_{\mu_k}(2^{-n_k}) \leq c 2^{-n_k} \log \frac{1}{w(2^{-n_k})}, \qquad k\geq 0
\end{equation}
and $\mu_\infty$ is supported on a set of finite $w$-entropy. Moreover, if $\mu$ does not charge any set of finite $w$-entropy then $\mu_\infty \equiv 0 $ for any choice of parameters $n_0$, $c$ and dyadic $w$-grid $\Dy_w$ as above.
\end{prop} 
Note that each decomposition provided by \thref{RobWdecomp} depends on the choices of parameters $n_0, c$ and the dyadic $w$-grid $\Dy_w$. This decomposition should be compared to the abstract decomposition given in \eqref{PCdecomp}, which somewhat provides a more elaborate description of the part of $\mu$ which does not charge sets of finite $w$-entropy. 

\begin{proof}[Proof of \thref{RobWdecomp}]
The proof is very similar to the decomposition provided in \cite{roberts1985cyclic}, which we shall adapt to our setting and provide details in certain steps which deviate from the mentioned work. Fix $n_0$, $c>0$ and let $\Dy_w = \cup_{k\geq 0} \Dy_{n_k}$ be a dyadic $w$-grid. We declare a dyadic arc $I \in \Dy_{n_0}$ to be \emph{light} for $\mu$ if 
\[
\mu(I) \leq c m(I) \log \frac{1}{w(m(I))} = c\, 2^{-n_{0}} \log \frac{1}{w(2^{-n_{0}})},
\]
and \emph{heavy} otherwise. Define the singular measure $\mu_0$ for Borel subsets $F\subseteq I \in \Dy_{n_{0}}$ by
\[
\mu_0(F) = \begin{cases}
    \mu(F) \qquad ,\text{if} \, \,  I \, \text{is light}, \\
    \frac{\mu(F)}{\mu(I)} c \, 2^{-n_0} \log \frac{1}{w(2^{-n_{0}})} \qquad ,\text{if} \, \, I \, \text{is heavy}. 
\end{cases}
\]
The measure $\mu_0$ constructed in this way is said to be the $\Dy_{n_{0}}$-grating of $\mu$. Observe that by construction $\mu-\mu_0$ is a non-negative singular measure, whose support lies in $H_0$: the union of the heavy intervals in $\Dy_{n_{0}}$. Moreover, one has
\[
\mu_0(I) = c \, 2^{-n_0} \log \frac{1}{w(2^{-n_0})}, \qquad \text{if} \, \, I \, \text{is heavy},
\]
and 
\[
\omega_{\mu_0}(2^{-n_0}) \leq c\, 2^{-n_0} \log \frac{1}{w(2^{-n_0})}.
\]
To proceed forward, the idea is now to define the measure $\mu_1$ as the $\Dy_{n_1}$-grating of $\mu-\mu_0$. In fact, iterating this procedure indefinitely by means of induction, we obtain a sequence of singular measures $\{\mu_k\}_{k\geq0}$, where each $\mu_{k}$ is a $\Dy_{n_{k}}$-grating of $\mu-\sum_{j=0}^{k-1} \mu_j$ and are supported inside $H_{k-1}$: the union of all heavy arcs in $\Dy_{n_{k-1}}$, and satisfy

\[
\mu_{k}(I) = c 2^{-n_{k}} \log \frac{1}{w(2^{-n_{k}})}, \qquad \text{if} \, \, I \, \text{is heavy},
\]
and 
\[
\omega_{\mu_{k}}(2^{-n_{k}}) \leq c \, 2^{-n_{k}} \log \frac{1}{w(2^{-n_{k}})}.
\]
Now set $\mu_\infty := \mu - \sum_{k\geq 0} \mu_k$ and note that by construction $H_k \supseteq H_{k+1}$ for all $k\geq 0$ and the difference $\mu - \sum_{j=0}^k \mu_j$ has its support inside $H_k$, thus $\mu_\infty$ has its support inside $F:= \cap_k H_k$. To verify that $F$ has Lebesgue measure $0$, observe that by the identity
\[
\mu_k(I) = c \, m(I) \log \frac{1}{w(m(I))}
\]
for heavy arcs $I$ in $\Dy_{n_k}$, we get
\begin{equation}\label{Fmes0}
c m(H_k) \log \frac{1}{w(2^{-n_k})} = \mu_k(H_k) \leq \mu_k(\T) \leq \mu(\T),
\end{equation}
hence $m(F) = \lim_k m(H_k) = 0$. Fix $k\geq 1$ and let $\mathcal{L}_k$ denote the set of interiors of the light arcs in $\Dy_{n_k}$, which are contained in $H_{k-1}$. Set $F_0 := \T \setminus \cup_k\cup_{J \in \mathcal{L}_k} J$. By definition $F_0$ is a closed set containing $F$. In fact, $F_0 \setminus F$ is a countable set which solely consists of endpoints of adjacent light arcs, hence $F_0$ also has Lebesgue measure zero. In order to show that $F$ has finite a $w$-entropy, it is sufficient to verify that $F_0$ has finite $w$-entropy, which on account of \thref{wsetcomp} amounts to showing that
\[
\sum_{k\geq 1} \sum_{J \in \mathcal{L}_k} m(J) \log \left( \frac{1}{w(m(J))} \right) < \infty.
\]
Set $L_k := \cup_{J \in \mathcal{L}_k} J$ and note that
\[
\sum_{k\geq 1} \sum_{J \in \mathcal{L}_k} m(J) \log \left( \frac{1}{w(m(J))} \right) = \sum_{k\geq 1} m(L_k) \log \left( \frac{1}{w(2^{-n_{k}})} \right) \leq \sum_{k\geq 2} m(H_{k-1}) \log \left( \frac{1}{w(2^{-n_{k}})} \right) .
\]
Now using the assumption that $\Dy_w = \cup_{k\geq 0} \Dy_{n_k}$ is a dyadic $w$-grid, there exists a  constant $\beta>0$, independent of $k$, such that 
\[
\log \left( \frac{1}{w(2^{-n_{k}})} \right) \leq \beta \log \left( \frac{1}{w(2^{-n_{k-1}})} \right),\qquad k\geq 1.
\]
Using this in conjunction with \eqref{Fmes0}, we obtain
\[
\sum_{k\geq 1} \sum_{J \in \mathcal{L}_k} m(J) \log \left( \frac{1}{w(m(J))} \right) \leq \frac{\beta}{c} \sum_{k\geq 1} \mu_k(\T) \leq \frac{\beta}{c} \mu(\T).
\]
This shows that $\mu_\infty$ is supported in the closed set $F_0$ which has finite $w$-entropy, hence the proof is complete.


\end{proof}
One can actually show that the description in \thref{RobWdecomp} characterizes singular measures $\mu$ which do not charge sets of finite $w$-entropy. This curious observation will not be relevant for our further developments and therefore we
refer the reader to the recent work in \cite{ivrii2022beurling} for results in this direction.

\subsection{Cyclicity via the Corona Theorem}
The following quantitative version of the classical Corona Theorem which appears in \cite{roberts1985cyclic} will be a crucial tool for our developments. 
\begin{thm}[Corona Theorem] \thlabel{Coronathm}

There exists an absolute constant $K>0$, such that whenever $f_1,f_2 \in H^\infty$ with $\norm{f_j}_{H^\infty} \leq 1$ for $j=1,2$, and 
\[
\inf_{z\in \D} \, \abs{f_1(z)} + \abs{f_2(z)}  \geq \delta,
\]
for some $0<\delta < 1/2$, then there exists $h_1,h_2 \in H^\infty$ with $\norm{h_j}_{H^\infty} \leq \delta^{-K}$ for $j=1,2$, such that the Bezout equation 
\[
f_1(z)h_1(z) + f_2(z)h_2(z) =1, \qquad z\in \D.
\]

\end{thm}
The initial step towards establish \thref{THM:Cycinner} is to consider the singular inner functions $S_{\mu_k}$, where $\mu_k$ is a singular measure appearing in the Roberts decomposition \thref{RobWdecomp} and find suitable Corona mates $f_k$ in the unit-ball of $H^\infty$, in such a way that the solutions $h_k,g_k$ enjoy the property that $1-S_{\mu_k}h_k = f_kg_k$ have very small $\G_w$-norm in terms of $n_k$. It turns out that certain monomials will do as Corona mates $f_k$ to $S_{\mu_k}$, hence we shall need a lemma which quantifies the asymptotic growth of the moments in $\G_w$.

\begin{lemma}\thlabel{z^nGW} For any majorant $w$ the $\G_w$-norm of the monomials enjoy the bound
\[
\norm{z^n}_{\G_w} \leq 3w(1/n)
\]
for any large positive integer $n$. 
\end{lemma}

\begin{proof}
Note that since $w$ is non-decreasing, we trivially have the estimate
\[
\sup_{1-|z|<1/n} w(1-|z|)|z|^n \leq w(1/n).
\]
Now let $\lambda>0$ such that $w^\lambda$ is a modulus of continuity and note that since $w^\lambda(t)/t$ is almost non-increasing, we have for any $1-|z|\geq 1/n$
\[
w(1-|z|)|z|^n \leq 2^{1/\lambda}\frac{w(1/n)}{(1/n)^{1/\lambda}} (1-|z|)^{1/\lambda}|z|^n.
\]
Since $(1-t)^{\lambda}t^n$ reaches its maximum at $t_n=n/(n+\lambda)$ on the unit-interval $[0,1]$, we obtain
\[
\sup_{1-|z|>1/n} w(1-|z|)|z|^n \leq 2^{1/\lambda} \left(\frac{n}{n+\lambda}\right)^{n+ 1/\lambda}   w(1/n) \leq 2 w(1/n).
\]
This proves the desired claim.
    
\end{proof}

%
%
\subsection{Proof of cyclicity in $\G_w$ } 
\label{cyclicity in Gw}
With the preparations in the previous subsections at hand, we now turn to our main task.

\begin{proof}[Proof of \thref{THM:Cycinner}]
Fix the parameter $c>0$ in such a way that $48cK <1$, where $K>0$ denotes the Corona constant appearing in \thref{Coronathm}, and let $n_0>0$ be an arbitrary positive integer with $w(2^{-n_0})^{12c} < \min(1/4, 3^{-1/K})$. Suppose $\mu$ is a positive finite singular Borel measure on $\T$ which does not charge any set of finite $w$-entropy. Applying \thref{ParameterLemma} and \thref{RobWdecomp} with parameters $c$, $n_0$ and $\{n_k\}_{k\geq 0}$ as above, we obtain the decomposition 
\[
\mu = \sum_{k\geq 0} \mu_k
\]
where each $\mu_k$ satisfies \eqref{modcontmuk}. Now recall that any positive finite Borel measure $\nu$ on $\T$ satisfies the following estimate
\begin{equation}\label{Mcbb}
\abs{S_{\nu}(z)} \geq \exp \left(- 6 \, \frac{\omega_{\nu}(1-|z|)}{1-|z|} \right) \qquad z\in \D.
\end{equation}
For instance, see Theorem 2 in \cite{anderson1991inner} for a simple proof of this fact. 
\proofpart{1}{Initial estimate using the Corona Theorem:}

Fix an arbitrary $k\geq 0$ and note that when $1-|z|\geq 2^{-n_k}$, the almost-decreasing property of a modulus of continuity gives 
\[
\frac{\omega_{\mu_k}(1-|z|)}{1-|z|} \leq 2 \frac{\omega_{\mu_k}(2^{-n_k})}{2^{-n_k}} \leq 2c\log \frac{1}{w(2^{-n_k})}.
\]
Using this in conjunction with \eqref{Mcbb} applied to $\nu=\mu_k$, we obtain the following bound from below:
\begin{equation}
    \abs{S_{\mu_k}(z)} \geq w(2^{-n_k})^{12c}, \qquad \qquad |z|\leq 1-2^{-n_k}.
\end{equation}
Now choosing the Corona-mate of $S_{\mu_k}$ to be the monomial $z^{2^{n_k}}$ it is easy to check that when $|z|>1-2^{-n_k}$ (here we utilize the  assumption that $w(2^{-n_0})^{12c} < 1/4$), one has
\[
|z|^{2^{n_k}} \geq (1-2^{-n_k})^{2^{n_k}} \geq 1/4 \geq w(2^{-n_k})^{12c}.
\]
This implies that
\[
\inf_{z\in \D} \, \abs{S_{\mu_k}(z)} + |z|^{2^{n_k}}   \geq w(2^{-n_k})^{12c},
\]
hence applying the Corona Theorem, we can find functions $g_k,h_k \in H^\infty$ enjoying the estimates $\norm{g_k}_{H^\infty}, \norm{h_k}_{H^\infty} \leq w(2^{-n_k})^{-12cK}$, which solve the equation $S_{\mu_k}h_k + z^{2^{n_k}} g_k =1$ on $\D$. Using this in conjunction with \thref{z^nGW} on moments in $\G_w$, we obtain
\begin{equation*}
\norm{S_{\mu_k}h_k -1}_{\G_w} \leq \norm{g_k}_{H^\infty} \norm{z^{2^{n_k}}}_{\G_w} \leq 3 w(2^{-n_k})^{1-12cK}.
\end{equation*}
In fact, utilizing the assumption on $c$ and $n_0$ once again, we actually get 
\begin{equation}\label{cyckest}
    \norm{S_{\mu_k}h_k -1}_{\G_w} \leq 3 w(2^{-n_k})^{1-12cK} \leq w(2^{-n_k})^{1-24cK} \leq w(2^{-n_k})^{24cK}.
\end{equation}
\proofpart{2}{The iterative procedure:}
The idea is now to combine the $S_{\mu_k}$'s all together and iterate the estimate in \eqref{cyckest}. To this end, fix a very large positive integer $N > 0$ and set $\nu_N := \sum_{k=0}^N \mu_k$. Denote by $h_k$ the Corona solution-pair associated to $S_{\mu_k}$ as in \eqref{cyckest} and set $H_N := h_{0} \cdot \cdot \cdot h_N$. With this at hand, we may combine the terms in the following way
\[
    1-S_{\nu_N}H_N = (1-S_{\mu_{0}}h_{0}) + S_{\mu_{0}}h_{0}(1-S_{\mu_{1}}h_{1}) +\dots + \left(\prod_{k=0}^{N-1} S_{\mu_k}h_k\right) (1-S_{\mu_N}h_N).
\]
Using this and applying \eqref{cyckest} in conjunction with the estimates $\norm{h_k}_{H^\infty} \leq w(2^{-n_k})^{-12cK}$ for each $k\geq 0$, we obtain
\[
\norm{S_{\nu_N}H_N -1}_{\G_w} \leq w(2^{-n_0})^{24cK} + \sum_{k=1}^N \left(\frac{w(2^{-n_k})^2}{w(2^{-n_{0}})\cdot \cdot \cdot w(2^{-n_{k-1}})}\right)^{12cK}, \qquad N\geq 1.
\]
Using \thref{ParameterLemma} we have
\[
\frac{w(2^{-n_k})^2}{w(2^{-n_{0}})\cdot \cdot \cdot  w(2^{-n_{k-1}})} \leq w(2^{-n_k}) \leq w(2^{-n_0})^{k+1}, \qquad k=1,2,\dots
\]
This in conjunction with the assumption on $w(1/n_0)^{12c} \leq 1/4$ provides us a constant $C>0$, independent of $n_0, c, N$, such that
\begin{equation}\label{N-normest}
\norm{S_{\nu_N}H_N -1}_{\G_w} \leq \sum_{k=0}^\infty w(2^{-n_0})^{(k+1)12cK} \leq C w(2^{-n_0})^{12cK}
\end{equation}
for all $N\geq 1$.

\proofpart{3}{The final limiting procedure:}
Now the observation that the above estimate in \eqref{N-normest} holds for any large $N > 0$ will be crucial in order to upgrade this estimate to hold for $\mu$ instead of each $\nu_N$. To this end, we shall need the following lemma, asserting that $H^\infty$ equipped with the weak-star topology inherited from $L^\infty(\T)$ is continuously embedded into $\G_w$.

\begin{lemma}\thlabel{HooincGw} Let $\{g_k\}_k \subset H^\infty$ be a sequence which converges uniformly on compacts to a bounded analytic function $g$ and $\sup_k \norm{g_k}_{H^\infty} < \infty$. Then $g_k$ converges to $g$ in the norm of $\G_w$. 
\end{lemma}
\begin{proof} Fix an arbitrary $\varepsilon>0$ and set $C:= \sup_{k} \norm{g_k-g}_{H^\infty}$, which is guaranteed to be finite by assumption. Since $w$ is non-decreasing, we have
\[
\sup_{1-|z|< \varepsilon} w(1-|z|)\abs{g_k(z)-g(z)} \leq Cw(\varepsilon).
\] 
But $g_k$ converges uniformly to $g$ on compacts of $\D$ thus
\[
\lim_{k \to \infty} \sup_{1-|z|\geq \varepsilon} w(1-|z|)\abs{g_k(z)-g(z)} = 0.
\]
Combining both these estimates, we arrive at 
\[
\limsup_{k \to \infty} \norm{g_k-g}_{\G_w} \leq Cw(\varepsilon).
\]
Letting $\varepsilon \to 0+$ finishes the proof.
\end{proof}
We now turn to the final part of the proof.
Note that for any large $N$, we have the following estimate:
\begin{equation*}\label{SmuHNest}
\norm{S_{\mu}H_N -1}_{\G_w} \leq \norm{S_{\nu_N}H_N - 1}_{\G_w} + \norm{S_{\mu-\nu_N}-1}_{\G_w} \leq C w(2^{-n_0})^{12cK} + \norm{S_{\mu-\nu_N}-1}_{\G_w}.
\end{equation*}
Now since the measures $\nu_N$ increase up to $\mu$ it is not difficult to prove that $S_{\mu-\nu_N}$ converges uniformly to $1$ on compacts and obviously has uniformly bounded $H^\infty$-norm, hence on account of \thref{HooincGw} we conclude that $S_{\mu-\nu_N}$ converges to $1$ in the norm of $\G_w$. Letting $N\to \infty$ in the equation above, we arrive at
\[
\inf_{h \in H^\infty} \norm{S_\mu h -1}_{\G_w} \leq C w(2^{-n_0})^{12cK}.
\]
Since $n_0>0$ can be chosen arbitrary large in the decomposition \thref{RobWdecomp} and in \thref{ParameterLemma}, we finally conclude that $S_\mu$ is cyclic in $\G_w$ and thus finishing the proof of \thref{THM:Cycinner}.

\end{proof}


\section{Proper shift invariant subspaces in $\G_w$} \label{SecPerm}
This section is chiefly devoted to proving that the singular inner function $S_{\mu_P}$ appearing in the decomposition \eqref{PCdecomp} satisfies the permanence property in $\G_w$.  

\begin{thm}\thlabel{THM:SmuPP}
Let $w$ be a majorant satisfying condition $(A_1)$ and let $\mu$ be a positive finite singular measure on $\T$. Then $S_{\mu_P}$ generates a proper $M_z$-invariant subspace in $\G_w$ satisfying the permanence property
\[
\left[ S_{\mu_P} \right]_{\G_w} \cap \mathcal{N}_+ \subseteq S_{\mu_P}\,  \mathcal{N}_+.
\]
\end{thm}
Our developments in this section are  influenced from the work in \cite{berman1984cyclic}, with certain adaptions and deviations suited for our general setting. 
%
%
\subsection{Identifying the singular part}\label{Singpart}
Let $\Omega \subseteq \D$ be a Jordan domain containing the origin, and let $\varphi: \D \to \Omega$ be a conformal mapping with $\varphi(0)=0$. For $0<p\leq \infty$ we recall that the Hardy space $H^p(\Omega)$ consists of analytic functions $f$ in $\Omega$ with
\[
\norm{f}_{H^p(\Omega)} := \sup_{0<r<1} \left( \int_{\T} \abs{(f\circ \varphi)(r\zeta)}^p dm(\zeta) \right)^{\min(1,1/p)} < \infty.
\]
In other words, $f$ belongs to $H^p(\Omega)$ if and only if $f\circ \varphi$ belongs to the classical Hardy space $H^p:= H^p(\D)$ on $\D$. In this setting, we always have the continuous embedding $H^p \subseteq H^p(\Omega)$. In this section, our main objective is to describe the singular inner factors in $H^p(\Omega)$ of functions in $H^p$, that is, the singular inner factor of $f\circ \varphi$ of a function $f\in H^p$. To this end, let $f \in H^p$ with Nevanlinna factorization $f= \mathcal{O}_f S_\mu B$. The following was observed in \cite{berman1984cyclic} which we insist on including for the sake of completeness. Recall that Beurling's theorem on $H^p$ implies that $\mathcal{O}_f$ is cyclic in $H^p$, hence on account of the continuous inclusion it is therefore also cyclic in $H^p(\Omega)$. However, an application of Beurling's theorem for simply connected domains implies that $\mathcal{O}_f$ must also be outer in $H^p(\Omega)$ and thus we conclude that $\mathcal{O}_f$ cannot give rise to any singular inner factor of $f$ in $H^p(\Omega)$. Since conformal maps lack singular inner factors and any Blaschke product $B$ on $\D$ is a product of normalized conformal self-maps of the unit disc, $B \circ \varphi$ cannot note give rise to a singular inner factor either. From this discussion, we conclude that the singular inner factor of $f\circ \varphi$ is located within the factor $S_\mu \circ \varphi$. Let $\zeta_0 \in \T$ and denote by $\sigma_{\zeta_0}$ the corresponding Aleksandrov-Clark measure of the conformal map $\varphi: \D \to \Omega$ with $\varphi(0)=0$, that is 
\[
 \frac{1-\abs{\varphi(z)}^2}{\abs{\zeta_0-\varphi(z)}^2} = \int_{\T} \frac{1-|z|^2}{|\zeta-z|^2} d\sigma_{\zeta_0}(\zeta) \qquad z\in \D.
\]
Now if $\zeta_0 \notin \varphi(\T)$, then it is not difficult to see that $d\sigma_{\zeta_0}$ is absolutely continuous wrt the Lebesgue arc-length measure $dm$ on $\T$ (for instance, see Lemma 2.1 in \cite{berman1984cyclic}). In other case, we have $\zeta_0= \varphi(\eta_0)$ for a unique $\eta_0 \in \T$ and one can then show that
\[
\frac{1-\abs{\varphi(z)}^2}{\abs{\zeta_0-\varphi(z)}^2}  = \frac{1-|z|^2}{|\varphi^{-1}(\zeta_0)-z|^2} \frac{1}{|\varphi'(\varphi^{-1}(\zeta_0))|} + \int_{\T} \frac{1-|z|^2}{|\zeta-z|^2} d\tau_{\zeta_0}(\zeta), \qquad z\in \D
\]
where $\tau_{\zeta_0}(\left\{\eta_0\right\})=0$. Here $\varphi'(\varphi^{-1}(\zeta_0))=\varphi'(\eta_0)$ denotes the \emph{angular derivative} of $\varphi$ at $\eta_0$, that is, the non-tangential limit of $\varphi'$ at $\zeta_0$, which is conventionally declared to be $\infty$ if it does not exist. Since $\partial \Omega$ is rectifiable, it follows from the F. and M. Riesz theorem that $\varphi' \in H^1$ and thus $\varphi$ has angular derivative at $dm$-a.e point in $\T$. As previously mentioned, one can also show that $d\tau_{\zeta_0}$ is absolutely continuous wrt $dm$. Now using the above expression, we can express
\[
- \log \abs{(S_{\mu} \circ \varphi) (z)} = \int_{\varphi(\T) } \frac{1-|z|^2}{|\varphi^{-1}(\zeta)-z|^2} \frac{d\mu(\zeta)}{|\varphi'(\varphi^{-1}(\zeta))|}  + \int_{\T\setminus \varphi(\T)} \int_{\T} \frac{1-|z|^2}{|\zeta-z|^2} d\tau_{\zeta}(\zeta) d\mu(\zeta).
\]
The discussion above indicates that singular part should be contained within the first term, which indeed was verified in Theorem 2.1 of \cite{berman1984cyclic}. In fact, carefully following the argument provided there one actually obtains the following precise description of the singular inner factor of $S_\mu \circ \varphi$, which we now phrase for future reference. 

\begin{thm}[Theorem 2.1, \cite{berman1984cyclic})]
\thlabel{singfac}
Let $\Omega \subseteq \D$ be a domain containing the origin and enclosed by a Jordan curve and let $\varphi: \D \to \Omega$ be a conformal mapping with $\varphi(0)=0$. Suppose $\mu$ is a positive finite singular Borel measure on $\T$ and denote by $S_\mu$ the associated singular inner function. Then the singular inner factor of $S_\mu \circ \varphi$ is given by 
\[
\exp \left( - \int_{\varphi^{-1}(\Gamma)} \frac{\zeta +z}{\zeta -z} \frac{1}{\abs{\varphi'(\zeta )}} d(\sigma \circ \varphi) (\zeta) \right), \qquad z\in \D
\]
where $\Gamma := \left\{ \zeta \in \partial \Omega \cap \T: \abs{\varphi'(\zeta)}< \infty \right\}$ and $(\sigma \circ \varphi) (E) := \sigma (\varphi(E))$ for Borel sets $E\subseteq \T$.  
\end{thm}
We also refer the reader to \cite{ivrii2022critical} for recent investigations on related topics.

\subsection{Construction of Privalov star-domains and embeddings into Hardy spaces} \label{Privstars1}
We shall now construct certain Privalov star-domains with sufficiently regular boundaries so that the corresponding conformal maps have angular derivatives everywhere. To this end, let $E \subset \T$ be a closed set of Lebesgue measure zero. We define the simply connected domain 
\begin{equation} \label{Privstar}
    \Omega_E := \left\{r\zeta\in \D: \zeta \in \T, \, \, 0\leq r < 1- h(\zeta) \right\}.
\end{equation}
where $h: \T \to [0,1/2]$ is a continuous function defined by 
\[
h(e^{it}) = \begin{cases} \frac{1}{2} \left(\frac{(t-a_k)(b_k-t)}{(b_k-a_k)} \right)^{2} \qquad ,e^{it}\in J_k \\
0 \qquad ,e^{it} \in E.

\end{cases}
\]
where $J_k= \{e^{it}: a_k < t < b_k\}$ is a connected component of $\T \setminus E$. Note that $h$ is comparable to
\[
h(\zeta) \asymp \textbf{dist}(\zeta,E)^{2} \qquad \zeta \in \T.
\]
Alternatively, $\Omega_E$ can be described as the simply connected domain enclosed by the Jordan curve $\gamma_E : \T \to \D$ defined by
\begin{equation}\label{gammaE}
    \gamma_E(\z) = \zeta \left(1-h(\zeta) \right), \qquad \zeta \in \T.
\end{equation}
By construction it is straightforward to verify that $\gamma_E$ is a $C^{1+\alpha}$-smooth Jordan curve with $C^\alpha$-smooth tangents for any $0<\alpha<1$ An application of Kellogg's Theorem (for instance, see Theorem 4.3 in \cite{garnett2005harmonic}) implies that any conformal map $\varphi: \D \to \Omega_E$ extends to a $C^{1+\alpha}$-smooth function on $\T$ with $\varphi' \neq 0$ on $\cD$, and similar conclusions hold for the inverse map $\varphi^{-1}: \Omega_E \to \D$. In what follows, we shall utilize the Privalov star-domains $\Omega_E$ constructed above in order to establish a continuous embedding from $\G_w$ into Hardy space $H^\infty(\Omega_E)$. It is here that the assumption $(A_1)$ will is crucially used. To this end, we declare that two real numbers $A,B$ satisfy $A \lesssim B$ if $A \leq cB$ for some constant $c>0$. Our first lemma concerns the construction of a Carleson-type function with certain prescribed properties.

\begin{lemma}\thlabel{Lem:Couter} Let $w$ be a majorant satisfying $(A_1)$ and suppose $E\subset \T$ is a set of finite $w$-entropy and $\Omega_E$ denotes the corresponding Privalov star-domain defined in \eqref{Privstar}. Then there exists an outer function $G_E : \D \to \D$ which extends continuously to $\T$ and satisfies the following estimate
\[
    \abs{G_E(z)} \leq  w(1-|z|), \qquad z\in \partial \Omega_E \cap \D.
\]
\end{lemma}
\begin{proof}
Let $I$ be a connected component of $\T \setminus E$ and $\{I_n\}_{n\in \mathbb{Z}}$ denote the collection of Whitney arcs associated to $I$ satisfying the properties 
\begin{equation*}\label{Whit}
    m(I_n) = \textbf{dist}(I_n,E) \asymp 2^{-|n|} m(I), \qquad n\in \mathbb{Z}
\end{equation*}
where the comparability is independent of $n$. Observe that the general property of majorants in \eqref{alm-dec} implies that for any connected component $I$ of $\T \setminus E$
\[
\log \frac{1}{w(m(I_n))} \lesssim   \abs{n} + \log \frac{1}{w(m(I))}  \qquad n\in \mathbb{Z}.
\]
Hence summing over all connected components $I$ of $\T \setminus E$ it readily follows that
\begin{equation}\label{WhitBC}
\sum_{I} \sum_{n\in \mathbb{Z}} m(I_n) \log \frac{1}{w(m(I_n))} \lesssim  \sum_I m(I)\left( 1+ \log \frac{1}{w(m(I))} \right)<\infty.
\end{equation}
For the sake of abbreviation, we now form the joint collection $\{J_k\}_k$ of all Whitney arcs corresponding to the family of connected components of $\T \setminus E$ and let $\xi_k$ denote the center of each Whitney arc $J_k$. For each $k$, we define the associated functions analytic functions
\begin{equation*}\label{psik}
\psi_k(z) := m(J_k) \log \frac{1}{w(m(J_k))} \frac{\xi_k}{\rho_k \xi_k - z}, \qquad z\in \D 
\end{equation*}
where $\rho_k = 1+m(J_k)$. Using these components, we build the corresponding Carleson function by defining
\begin{equation}\label{defGE}
G_E (z) := \exp \left( - N \sum_k \psi_k(z) \right) \qquad z\in \D.
\end{equation}
where $N>0$ is a positive number to be chosen later. 
Note that the $w$-entropy condition on $E$ in conjunction with the observation in \eqref{WhitBC} ensures that $\sum_k \psi_k(z)$ converges uniformly on compact subsets of $\D$ and thus $G_E$ is a well-defined analytic function there. Furthermore, we also note that
\[
\Re \left( \psi_k(z) \right) = m(J_k) \log \frac{1}{w(m(J_k))} \frac{\rho_k - \Re(z\conj{\xi_k})}{\abs{\rho_k \xi_k -z }^2} >0, \qquad z\in \D
\]
which implies that $G_E : \D \to \D$ is outer. In fact, since the poles $\{\rho_k \xi_k \}_k$ of $\sum_k \psi_k$ accumulate only at $E$, $G_E$ actually extends analytically across each Whitney arc $J_k$. Furthermore, a straightforward argument involving Taylor expansions shows that $\Re(\xi_k /(\rho_k \xi_k -\zeta)) \asymp 1/m(J_k)$ whenever $\zeta \in J_k$ with comparability constants independent of $k$. Using this in conjunction with the property of the Whitney arcs, we readily deduce that
\[
\abs{G_E(\zeta)} \lesssim w( \textbf{dist}(\zeta, E) )^{cN}, \qquad \zeta \in \T \setminus E
\]
for some absolute constant $c>0$. This shows implies that $G_E$ extends continuously to $\T$ with $G_E = 0 $ on $E$. It now remains to verify the radial estimate of $G_E$ on $\partial \Omega_E \cap \D$. Fix an arbitrary point $z = \abs{z}\zeta \in \partial \Omega_E \cap \D$ and let $J_k$ denote the unique Whitney arc containing the radial projection $\zeta \in \T$ of $z$. In that case, we recall that by construction of the Privalov star-domain $\Omega_E$ we have
\begin{equation*}\label{h*est}
1-|z| = h(\zeta) \asymp \text{dist}(\zeta ,E)^{2} \asymp m(J_k)^2 \qquad z \in \partial \Omega_E \cap \D, \, \, \zeta \in J_k.
\end{equation*}
As previously mentioned, a similar computation involving Taylor expansions shows that 
\[
\Re \left( \frac{\xi_k}{\rho_k \xi_k -z } \right)  \asymp  1/m(J_k), \qquad z\in \partial \Omega_E \cap \D, \, \, \zeta \in J_k.
\]
Using this in conjunction with the assumption $(A_1)$ (here it is crucially used) on the majorant $w$, we obtain
\[
\abs{G_E(z)} \lesssim w \left( \text{dist}(\zeta ,E) \right)^{cN} \lesssim w \left( 1-|z| \right)^{c'N}, \qquad z\in \partial \Omega_E \cap \D
\]
where the constant $c, c'>0$ are absolute constants. It is now only a matter of choosing $N>0$ sufficiently large. 
\end{proof}

An important consequence of \thref{Lem:Couter} is the following result which we phrase for future reference.
\begin{prop}\thlabel{GwHardyEmb}
Let $M_{E}(g)(z) = G_E(z) g(z)$ denote the multiplication operator by Carleson outer function $G_E$ defined in \eqref{defGE}. Then the linear map $M_E$ maps $\G_w$ into $H^\infty(\Omega_E)$
continuously. 
\end{prop}
\begin{proof}
Let $Q$ be an analytic polynomial and observe that by \thref{Lem:Couter}, we have that 
\[
\sup_{z \in \partial \Omega_E \cap \D } \abs{G_E(z) Q(z)} \leq \norm{Q}_{\G_w}.
\]
Now since $G_E$ continuous on $\T$, the maximum principle implies that 
\[
\sup_{z \in \Omega_E  } \abs{G_E(z) Q(z)} \leq \norm{Q}_{\G_w}.
\]
Using the fact that the analytic polynomials are dense in $\G_w$, it follows that $M_E (g) = G_Eg$ uniquely extends to a continuous linear operator on all of $\G_w$, mapping into $H^\infty (\Omega_E )$.
\end{proof}

\subsection{The permanence property}
Combining the tools developed in the previous subsections, we now turn to the proof of \thref{THM:SmuPP}

\begin{proof}[Proof of \thref{THM:SmuPP}]
We shall divide the proof in two different steps.
\proofpart{1}{Support on a single set of finite $w$-entropy:}

We first show that if $\mu$ is a positive finite singular Borel measure supported on a single set $E\subset \T$ of finite $w$-entropy, then $S_{\mu}$ satisfies the permanence property in $\G_w$, that is 
\[
\left[ S_\mu \right]_{\G_w} \cap \N_+ \subseteq S_\mu \, \N_+.
\]
Let $f\in \left[ S_\nu \right]_{\G_w} \cap \N_+$ be an arbitrary element and pick a sequence of analytic polynomials $Q_n$ such that $S_\nu Q_n \to f$ in $\G_w$. Writing $f= u/v$ with $u,v \in H^\infty$ and $v$ outer, and using  the fact that $H^\infty$ is a multiplier of $\G_w$ we have that $vS_\nu Q_n \to u$ in $\G_w$. Now applying \thref{GwHardyEmb} we get
\[
Q_n(z)G_E(z)v(z) S_\nu(z) \to G_E(z) u(z)
\]
converges uniformly on $\Omega_E$, and thus also in the norm of $H^2(\Omega_E)$. However, this means that the function $G_E u$ belongs to the $M_z$-invariant subspace of $H^2(\Omega)$ generated by $G_E v S_\nu$. Recalling that $G_E$ and $v$ are outer functions in $H^\infty$, we may apply Beurling's theorem in the context of $H^2(\Omega_E)$ to conclude that $u\circ \varphi$ is divisible by the singular inner factor $S_{\widetilde{\mu}}$ of $S_\mu \circ \varphi$, explicitly given in \thref{singfac}. Composing with $\varphi^{-1}$ to the right, we get that $u$ is divisible by the singular inner factor of $S_{\widetilde{\mu}} \circ \varphi^{-1}$. Now applying \thref{singfac} to $\widetilde{\mu}$ with $\varphi^{-1}$ replacing $\varphi$, we actually get that $u$ is divisible by $S_\mu$. This is enough to establish the desired claim. 
\proofpart{2}{Unions of sets with finite $w$-entropy:}
Let $E= \cup_n E_n$ be a carrier set of $\mu_P$, where $E_n \subset E_{n+1}$ and each $E_n$ have finite $w$-entropy. For the sake of abbreviation, set $\mu_n = \mu_P \lvert_{E_n}$ and note that the argument provided in the previous step actually shows that 
each singular inner function $S_{\mu_n}$ satisfies the following property:
\[
\left[ S_{\mu_n}\right]_{\G_w} \cap H^\infty  \subseteq S_{\mu_n} H^\infty.
\]
Now pick an arbitrary element $f\in \left[ S_{\mu_p}\right]_{\G_w} \cap \N_+ $ and write $f=u/v$ with $u,v \in H^\infty$ and $v$ outer. It is straightforward to verify that $u \in \left[ S_{\mu_p}\right]_{\G_w} \cap H^\infty$, hence combining with the fact that each $S_{\mu_n}$ divides $S_{\mu_P}$, we obtain the following containment
\[
u\in \left[ S_{\mu_p}\right]_{\G_w} \cap H^\infty  \subseteq \bigcap_n \left[ S_{\mu_n}\right]_{\G_w}\cap H^\infty  \subseteq \bigcap_n S_{\mu_n} H^\infty
\]
As a consequence, for any $n$, there exists $h_n \in H^\infty$ such that $u= S_{\mu_n}h_n$. Now since 
\[
\sup_n \norm{h_n}_{H^\infty} = \norm{u}_{H^\infty} < \infty
\]
invoking Helly's selection theorem allows us to extract a subsequence $\{h_{n_k}\}_k$ which converges in weak-star to some $h \in H^\infty$. But since the sequence $\{\mu_n\}_n$ increases up to $\mu_P$, so does any subseqence and we therefore conclude that $u= S_{\mu_P}h$. This shows that $S_{\mu_P}$ satisfies the permanence property and we conclude the proof of the theorem. 
    
\end{proof}

\section{Proof of main results}

Here we summarize our developments in the previous sections in order to deduce \thref{GwNevMz} and \thref{Thm:GenKR}.

\subsection{An extension of the Korenblum-Roberts Theorem}

With all necessary ingredients at hand our proof now becomes fairly short.

\begin{proof}[Proof of \thref{Thm:GenKR}]
We shall divide the proof in three simple steps.
\proofpart{1}{Proof of cyclicity:} 
Observe that in order to deduce the claim in $(ii)$, it suffices to establish the equality of sets
\begin{equation}\label{cycsets}
\left[ \mathcal{O}_f S_{\mu_C} \right]_{\G_w} = \left[ S_{\mu_C} \right]_{\G_w}.
\end{equation}
That done, we then invoke \thref{THM:Cycinner} to conclude that $\mathcal{O}_f S_{\mu_C}$ is cyclic in $\G_w$. Note that since the containment $\subseteq $ vacuously holds it remains to establish $\supseteq$. To this end, we shall utilize Theorem 7.4 from Chap. II in \cite{garnett}, which asserts that there exists a sequence of bounded analytic functions $\{F_n\}_n$ satisfying the following properties:
\begin{enumerate}
    \item[(a.)] $\sup_n \norm{F_n \mathcal{O}_f}_{H^\infty} \leq 1$.
    \item[(b.)]  $\lim_{n}F_n(\zeta)\mathcal{O}_f(\zeta)=1$ for $dm$-a.e $\zeta \in \T$.
\end{enumerate}
In other words, outer functions are weak-star (sequentially) cyclic in $H^\infty$. Now by means of multiplying with $S_{\mu_C}$ and taking Poisson extensions we actually conclude that $F_n\mathcal{O}_fS_{\mu_C}$ converges to $S_{\mu_C}$ weak-star in $H^\infty$. An application of \thref{HooincGw} then shows that the convergence also holds in $\G_w$ and thus we conclude that $S_{\mu_C} \in \left[\mathcal{O}_fS_{\mu_C} \right]_{\G_w}$. This proves \eqref{cycsets} which together with \thref{THM:Cycinner} establishes part $(ii)$.

\proofpart{2}{The permanence property:}
Recall that \thref{THM:SmuPP} asserts that $S_{\mu_P}$ satisfies the permanence property. Of course, by multiplying $S_{\mu_P}$ with any Blashcke product $B$ the permanence property remains intact. In fact, this is a simple consequence of the fact that normal families preserve zeros inside $\D$. This completes the proof of $(i)$.
\proofpart{3}{Showing that $\left[f\right]_{\G_w} = \left[B S_{\mu_P}\right]_{\G_w}$:} The inclusion $\subseteq$ is trivial, and for the reverse inclusion we may on account of part $(ii)$ pick a sequence of analytic polynomials $\{Q_n\}_n$ such that $FS_{\mu_C}Q_n$ converges to $1$ in $\G_w$. Multiplying with $BS_{\mu_P} \in H^\infty$, we conclude that $fQ_n$ converge to $BS_{\mu_P}$ in $\G_w$. This gives the reverse containment and the proof is therefore complete.
\end{proof}

\subsection{Shift invariant subspaces generated by Nevanlinna functions}

With \thref{Thm:GenKR} at hand, we now turn to the proof of \thref{GwNevMz}. The following proposition will play the principal role towards this end.

\begin{prop}\thlabel{fS=f} Let $f \in \N \cap \G_w$ with Nevanlinna factorization $f= \mathcal{O}_f \left( S_\mu / S_\nu \right) B_{\Lambda}$, where $\mu, \nu$ are positive mutually singular measures. Then $S_\nu$ is cyclic in $\G_w$ and 
\[
\left[ f\right]_{\G_w} = \left[ f S_\nu \right]_{\G_w}.
\]
    
\end{prop}

\begin{proof}
Here we establish the proof in two steps.
\proofpart{1}{Cyclicity of $S_\nu$:}
We shall first establish that $S_\nu$ is cyclic $\G_w$. Recall that we may express $\mathcal{O}_f = \mathcal{O}_1 /\mathcal{O}_2$ with $\mathcal{O}_j \in H^\infty$ outer for $j=1,2$. With this at hand, we can pick a sequence of analytic polynomials $\{Q_n\}_n$ which converge to $f$ in $\G_w$. Multiplying by $S_{\nu}\mathcal{O}_2 \in H^\infty$ and applying \thref{Thm:GenKR} we conclude that $\mathcal{O}_1S_\mu B_\Lambda$ belongs to the shift invariant subspace 
$\left[\mathcal{O}_2 S_\nu\right]_{\G_w} = \left[S_{\nu_P}\right]_{\G_w}$. But then permanence property in $(ii)$ of \thref{Thm:GenKR} implies that $S_{\nu_P}$ divides $S_{\mu}$, which by the assumption of $\mu,\nu$ being mutually singular forces $\nu_P \equiv 0$. This shows that $\nu= \nu_C$ and hence according to part $(ii) of $\thref{THM:Cycinner} we conclude that $S_{\nu}$ is cyclic in $\G_w$. 

\proofpart{2}{Description of $\left[f\right]_{\G_w}$:}
We now devote our attention to showing that 
\[
\left[ f \right]_{\G_w} = \left[ f S_{\nu} \right]_{\G_w}.
\]
The inclusion $\supseteq$ vacuously holds, hence we only need to establish that $\subseteq$. To this end, we shall truncate $f$ by heigth, that is, we fix a number $N>0$ and define the outer function 
\[
\abs{\mathcal{O}_N(\zeta)} = \min \left( \, \abs{\mathcal{O}_f (\zeta)}\, , N \right) \qquad \zeta \in \T.
\]
Considering the function $g_N := \mathcal{O}_N S_{\mu} B_{\Lambda} \in H^\infty$, the principal step towards our goal is to deduce that
\begin{equation}\label{g_Ns}
    g_N / S_{\nu} \in \left[ g_N \right]_{\G_w}, \qquad N>0.
\end{equation}
In order to prove this claim, we shall find inspiration from the technical lemma in \cite{korenblum1981cyclic}. To this end, consider the set of real numbers
\[
M = \left\{ 0\leq t\leq 1: g_N/S^t_\nu \in \left [ g_N \right]_{\G_w} \right\}
\]
and note that $0 \in M$ by definition. Observe that since 
\[
\abs{ \frac{g_N(z)}{S^t_\nu(z)}} \leq \abs{\frac{g_N(z)}{S_\nu(z)}} \qquad z\in \D, \, \,  0<t<1.
\]
it follows that $\phi(t)= g_N/S^t_\nu $ is continuous from $[0,1]$ into $\G_w$ and hence $M$ is closed. Let $t_0 = \max(M)$ and suppose for the sake of obtaining a contradiction that $t_0<1$. Pick a number $0<\delta <1-t_0$ and recall that since $S_{\nu}$ is cyclic in $\G_w$ it follows from \thref{Thm:GenKR} that $S^{\delta}_{\nu}$ is cyclic in $\G_{w^{t}}$ for any $t>0$. Moreover, observe that $g_N / S_{\nu}^{t}$ belongs to $\G_{w^t}$ for any $0<t<1$:
\[
\sup_{z\in \D} w(1-|z|)^t \abs{g_N(z) / S_{\nu}^{t}(z)} \leq \norm{g_N}^{1-t}_{H^\infty} \norm{g_N / S_{\nu}}^t_{\G_w}.
\]
Combining these observation, we may pick a sequence of analytic polynomials $\{Q_n\}_n$ such that  
\[
S_{\nu}^{\delta}Q_n \to 1\qquad \text{in} \, \, \G_{w^{1-t_0-\delta}},
\]
and then by multiplying with $g_N/S^{t_0 + \delta}_{\nu} \in \G_{w^{t_0+\delta}}$ we obtain
\[
(g_N/S^{t_0}_{\nu}) Q_n  \to g_N/S^{t_0 + \delta}_{\nu} \qquad \text{in} \, \, \G_{w}.
\]
Now since $g_N/S_{\nu}^{t_0} \in \left[ g_N \right]_{\G_w}$ by the assumption $t_0 \in M$, the above argument shows that $\max(M)<t_0 + \delta \in M$. This contradiction shows that $t_0=1$ and thus we conclude that \eqref{g_Ns} holds. That done, we note that the function $g := fS_\nu = \mathcal{O}_f S_\mu B_{\Lambda} \in \N^+$ by construction satisfies $\norm{g_N /g}_{H^\infty}\leq 1$, hence $\left[ g_N \right]_{\G_w} \subseteq \left[ g \right]_{\G_w}$ for all $N>0$. Using this in conjunction with the observation that $g_N / S_{\nu}$ converge to $g/S_\nu = f$ in $\G_w$, we conclude from \eqref{g_Ns} that 
\[
f= g / S_{\nu} \in \left[ g \right]_{\G_w} = \left[ f S_\nu \right]_{\G_w}.
\]
This completes the proof of the lemma.

\end{proof}
We now complete the proof of our main Theorem.
\begin{proof}[Proof of \thref{GwNevMz}]
Note that cyclicity of $S_{\nu}$ in $(ii)$ follows from \thref{fS=f}, and according to \thref{Thm:GenKR} the claim in $(i)$ is proved once we establish
\begin{equation}\label{f=BsmP}
\left[ f \right]_{\G_w} = \left[ B_\Lambda S_{\mu_P} \right]_{\G_w}.
\end{equation}
According to \thref{fS=f} we may actually assume that $f\in \G_w \cap \N^+$ with Nevanlinna factorization $f=\mathcal{O}_f S_{\mu} B_\Lambda$, and recall that we may also express $\mathcal{O}_f = \mathcal{O}_1 /\mathcal{O}_2$ with $\mathcal{O}_j \in H^\infty$ outer for $j=1,2$. Note first that $\mathcal{O}_2 f = \mathcal{O}_1S_{\mu}B$ implies that $\mathcal{O}_1 S_{\mu}B_\Lambda \in \left[ f \right]_{\G_w}$, hence an application of $(i)$ from \thref{Thm:GenKR} gives the inclusion
\begin{equation} \label{BSinf}
\left[ B_\Lambda S_{\mu_P} \right]_{\G_w} = \left[ \mathcal{O}_1S_{\mu}B_\Lambda \right]_{\G_w} \subseteq \left[ f \right]_{\G_w}.
\end{equation}
For the reverse containment, we truncate $f$ by height, that is, we define the outer function $\mathcal{O}_N$ by
\[
\abs{\mathcal{O}_N(\zeta)} = \min \left( \, \abs{\mathcal{O}_f(\zeta)}\, , N \right) \qquad \zeta \in \T,
\]
and note that
\[
\abs{\mathcal{O}_N (z)} \leq \abs{\mathcal{O}_f(z)} \qquad z\in \D.
\]
Consider the $H^\infty$-functions $f_N := \mathcal{O}_N S_\mu B_\Lambda$ and note that $f_N \in \left[S_\mu B_\Lambda \right]_{\G_w}$ for each $N$. Using the fact that $f_N$ converges to $f$ in $\G_w$ in conjunction with part $(i)$ of \thref{Thm:GenKR} we conclude that 
\[
f\in \left[B_\Lambda S_\mu \right]_{\G_w} = \left[B_\Lambda S_{\mu_P} \right]_{\G_w}.
\]
This completes the proof.
\end{proof}

%
%
\subsection{Certain improvements and technical remarks}\label{ImprovedA} 
Here we shall briefly discuss and make a few remarks on certain achievable improvements of our results and suggest some further directions of works. Firstly, our strategical decision for restriction our attention to weights $w$ which do not decay faster than polynomials is for the purpose of providing better clarity in exposition. In fact, weights which decay as fast as 
\begin{equation}\label{explogclass}
w(t)= \exp(-\alpha \log^{\beta}(e/t)),\qquad  \alpha,\beta>0
\end{equation}
can also be handled, at the cost of adding the following two assumptions on top of $(A_1)$ to our class of weights.
\begin{enumerate}
    \item[(a)] There exists constants $\kappa, C_1>0$ such that
    \[
\sup_{0<t<1} t^n w(1-t) \leq C_1 w(1/n)^{\kappa}, \qquad n=1,2, 3, \dots
\]
\item[(b)]  There exists a constant $C_2>0$ such that 
\[
\int_0^{\ell} \log \frac{1}{w(t)} dt \leq C_2 \ell \log \frac{1}{w(\ell)}, \qquad \ell >0.
\]
\end{enumerate}
Condition $(a)$ provides the asymptotic decay on moments in the $\G_w$-norm, which seems indispensable in the argument by J. Roberts involving the Corona Theorem for proving cyclicity of inner functions. The condition falls short once we go beyond the class of weights in \eqref{explogclass}, hence a different approach involving \emph{linear programming} inspired from the work of B. Korenblum looks more promising in allowing us to further extend our results in a certain direction. Condition $(b)$ merely ensures that we can go back and forth between the different descriptions of the notions on $w$-entropy and puts vastly less restriction on the rate of decay on $w$. 

Our second remark concerns the utility of condition $(A_1)$, which is notably applied in Section 3. There our purpose was to construct special Privalov-star domains $\Omega_E$ associated to sets $E$ of finite $w$-entropy satisfying $\partial \Omega_E \cap \T = E$, with the intent of accomplishing the following conflicting tasks:

\begin{description}
   \item[(i)] $\partial \Omega_E$ is smooth enough so that any conformal map $\varphi: \D \to \Omega_E$ has finite non-zero angular derivatives in $E$.
   \item[(ii)] $\left(G_E g\right) \circ \varphi \in H^2(\D)$, for all conformal map $\varphi: \D \to \Omega_E$.
\end{description}
Note that condition $(ii)$ implies that the singular inner factor $S_{\mu}\circ \varphi$ is as large as possible, see \thref{singfac}, which in combination with $\textbf{(ii)}$ allows us to obtain the sharp dichotomy between cyclicity and the permanence principle of singular inner functions. The more regular $\Omega_E$ is, the harder condition $\textbf{(ii)}$ is to be fulfilled, without any additional assumptions on $w$. In particular, if the decay of $w$ is sufficiently fast then more precise estimates involving harmonic measures seems to be required in order to deduce $\textbf{(ii)}$. In subsection 3.2 we utilized the classical Kellogg's Theorem to "smooth out" the corners of typical Privalov star-domains using fairly elementary means. For more sophisticated constructions of Privalov stars with other intended purposes, we refer the reader to recent works in \cite{ivrii2022beurling}.

We now remark on some related questions that arose in our developments. As briefly mentioned in the introduction, there is a certain threshold on the rate of decay of weights $w$ for which notions of entropy no longer play a decisive role in problems on cyclicity of inner functions. The following result stems from early works carried out independently by T. Carleman, M. Keldysh, A. Beurling and N. Nikolskii, asserting that a weight $w$ satisfying certain mild regularity assumptions fulfills the condition
\begin{equation} \label{BKcond}
\int_0^1 \sqrt{ \frac{\abs{\log w(t)}}{t}} dt = \infty
\end{equation}
if and only if any singular inner function is cyclic in the corresponding growth space $\G_w$. See  \cite{borichev2014cyclicity}, \cite{el2012cyclicity} for a historic description on the problem and for recent improvements of the mentioned result. A natural class of weights which are not covered by our methods and for which condition \eqref{BKcond} above fails is given by $w(t)= \exp(-t^{\alpha-1})$ with $0<\alpha <1$, where the corresponding sets of finite $w$-entropy, usually referred to as $\alpha$-Beurling-Carleson sets, are closed sets $E \subset \T$ of Lebesgue measure zero satisfying 
\[
\sum_k m(I_k)^{\alpha} < \infty.
\]
where $\{I_k\}_k$ denote the connected components of $\T \setminus E$. It is our intention to further investigate this matters in upcoming work, and we refer the reader to investigations of related kind in \cite{ivrii2022beurling} and in \cite{malman2023revisiting}.

\section{Regular functions in model spaces}

This section is devoted to establishing our main application announced in \thref{THM:Model}. Throughout this subsection, we shall assume that $w$ is a modulus of continuity satisfying condition $(A_2)$ and let $0<\alpha<1$ designate the specific constant appearing in therein.

\subsection{Cauchy dual reformulations in terms of model spaces}
We shall need to identify the Cauchy dual $(\G_w)'$ of the growth space $\G_w$, that is, we would like to identify the continuous linear functionals on $\G_w$ as a certain class of analytic functions in $\D$ considered in the Cauchy $H^2$-pairing on $\T$. In other words, we seek analytic functions $f$ in $\D$ for which the limit
\begin{equation}\label{Cauchydual}
\ell_f(g) := \lim_{r\to1-} \int_{\T} g(r\zeta) \conj{f(r\zeta)} \, dm(\zeta)
\end{equation}
exists for any $g\in \G_w$ and defines a bounded linear functional on $\G_w$. To this end, note that if $f$ is analytic in $\D$ and $g$ is an element in $\G_w$, an application of Green's formula gives
\[
\int_{\T} g(r\zeta) \conj{f(r\zeta)} \, dm(\zeta) = g(0) \conj{f(0)} + r\int_{\D}g(rz) \conj{f'(rz)} dA(z), \qquad 0<r<1.
\]
We claim that any analytic function $f$ in $\D$ satisfying 
\[
\norm{f}_{\mathcal{F}_w}:= \abs{f(0)} + \int_{\D} \abs{f'(z)}\frac{dA(z)}{w(1-|z|)} < \infty
\]
induces a bounded linear functional on $\G_w$ in the Cauchy $H^2$-pairing given by \eqref{Cauchydual}. We shall denote this class of functions by $\mathcal{F}_w$. To see this, fix an $f\in \mathcal{F}_w$ and apply Green's formula to the difference of dilations $g_{r} -g_{\rho}$ with $g\in \G_w$ to obtain 
\[
\abs{\int_{\T} \left(g(r\zeta)-g(\rho \zeta) \right) \conj{f(r\zeta)} \, dm(\zeta) } \leq C \norm{g_{r}-g_{\rho}}_{\G_w} \norm{f}_{\mathcal{F}_w}.
\]
Since dilations are continuous in $\G_w$, the Cauchy criterion shows that $\ell_f(g)$ exists for any $f\in \G_w$ and whenever $f\in \mathcal{F}_w$. Summarizing we conclude that 
\[
\mathcal{F}_w \subseteq \left( \G_w \right)'.
\]
This inclusion will be enough for our purposes, but one can in fact show that $\mathcal{F}_w$ is isomorphic to $\left( \G_w \right)'$ with equivalent norms. Now the next results allows us to rephrase the problem of cyclicity for inner functions $\Theta$ and the existence of certain functions in the corresponding model spaces $K_\Theta$.

\begin{lemma} \thlabel{cyclicmodel}
Let $X \supset H^\infty$ be Banach space for which the analytic polynomials are dense, and let $X'$ denote its Cauchy dual. Then an inner function $\Theta$ is cyclic in $X$ if and only if $X' \cap K_{\Theta}=\{0\}$.
    
\end{lemma}

\begin{proof} 
We primarily note that $X \supset H^\infty$ implies that $X' \subseteq H^2$, hence Cauchy-dual elements of $X$ have well-defined boundary values $m$-a.e on $\T$. Now if $\Theta$ is cyclic in $X$, then whenever $f \in X'$ with 
\begin{equation} \label{KThetaX'}
    \int_{\T}f(\zeta) \conj{\Theta(\zeta) \zeta^n} dm(\zeta) = 0, \qquad \forall n\geq 0,
\end{equation}
we must have $f\equiv 0$, which is equivalent to $X' \cap K_\Theta =\{0\}$. On the other hand, if $\Theta$ is not cyclic in $X$, then by the Hahn-Banach separation theorem, there exists a non-trivial function $f\in X'$ satisfying \eqref{KThetaX'}. But the containment $X' \subseteq H^2$ in conjunction with the F. and M. Riesz Theorem readily implies that $f\in K_\Theta$, and consequently the intersection $X' \cap K_\Theta$ contains $f$ and is thus non-trivial.

\end{proof}


The next lemma in line is a special case of a general result in the theory of reproducing kernel Hilbert spaces, which will be useful for our purposes. For the sake of completeness, we include a brief sketch of proof. 

\begin{lemma}\thlabel{absdens} Let $\{\Theta_n\}_{n}$ be a sequence of inner functions with the properties that each $\Theta_n$ is a divisor of $\Theta_{n+1}$ and $\lim_{n\to \infty} \Theta_n(z)= \Theta(z)$ for all $z\in \D$. Then $\cup_{n} K_{\Theta_n}$ is a dense subset of $K_{\Theta}$. Moreover, if $ S \subseteq H^2$ is a set with the property that $S \cap K_{\Theta_n}$ is dense in $K_{\Theta_n}$ for each $n$, then $S \cap K_{\Theta}$ is dense in $K_{\Theta}$.
\end{lemma}

\begin{proof}[Proof-sketch:] Since each $\Theta_n$ divides $\Theta_{n+1}$ and $\Theta_n \to \Theta$ converges pointwise in $\D$, we have the chain of inclusions $K_{\Theta_n} \subseteq K_{\Theta_{n+1}} \subseteq K_{\Theta}$, hence we conclude that 
\[
\bigcup_n K_{\Theta_n} \subseteq K_{\Theta}.
\]
Note that any reproducing kernel in $K_{\Theta}$ can be approximated by a reproducing kernels in $K_{\Theta_n}$ (with varying $n$) pointwise in $\D$ and with uniformly bounded $H^2$-norm, hence also weakly in $K_{\Theta}$. An argument involving Banach-Saks shows that a convex combination of the corresponding reproducing kernels converge in $H^2$-norm, which is enough to deduce that $\cup_n K_{\Theta_n}$ is dense $K_\Theta$ in natural manner. The second statement is now a straightforward consequence from the first. 
    
\end{proof}

\subsection{Regularity properties of functions in $\A_w$}

Roughly speaking, the following lemma shows that if $w$ is a modulus of continuity for which some power $w^\alpha$ with $\alpha>0$ satisfies the Dini-condition, then the Cauchy projection does not distort the modulus of continuity too much. 

\begin{lemma} \thlabel{P+reg}
Let $w$ be a modulus of continuity and suppose that there exists a constant $\alpha >0$ such that $w^{1+\alpha}$ is a modulus of continuity and 
\[
\int_0^1 w^{\alpha}(t) \frac{dt}{t} <\infty, 
\]
Then the Cauchy projection $P_+$ maps $C_{w^{\alpha+1}}(\T)$ continuously into $\A_w$.
\end{lemma} 
\begin{proof} Recall that the Cauchy projection $P_+$ maps $C_{w^{\alpha+1}}(\T)$ into $\A_{\widetilde{w}}$, where $\widetilde{w}$ is a majorant satisfying
\begin{equation*}\label{regmaj}
   \widetilde{w}(\delta) \leq C\int_0^{\delta} \frac{w^{\alpha+1}(t)}{t} dt + \delta \int_{\delta}^1 \frac{w^{\alpha+1}(t)}{t^2}dt,\qquad 0\leq \delta \leq 1
\end{equation*}
for some constant $C>0$ independent of $\delta$ (for instance, see Theorem 1.3, Chap. III of \cite{garnett}). It immediately follows from the Dini-condition of $w^\alpha$ and the assumption $w(t)/t$ being almost-decreasing that
\[
\int_{0}^{\delta} \frac{w^{\alpha+1}(t)}{t} dt + \delta \int_{\delta}^1 \frac{w^{\alpha+1}(t)}{t^2}dt \leq w(\delta)  \int_{0}^{1}  \frac{w^{\alpha}(t)}{t}dt  =: C' w(\delta).
\]
This shows the inclusion $\A_{\widetilde{w}} \subseteq \A_{w}$ and thus $P_+: C_{w^{\alpha+1}}(\T) \to \A_{w}$ continuously.
\end{proof}

Next we shall need a lemma which asserts that the derivative of a function in $\A_w$ satisfies a certain radial growth restriction. 

\begin{lemma} \thlabel{AwinBw}
Let $w$ be a modulus of continuity. Then there exists a universal constant $C>0$, such that any function $f\in \A_w$ enjoys the growth restriction
\begin{equation} \label{Bwcond}
\abs{f'(z)} \leq C \norm{f}_{\A_w} \frac{w(1-|z|)}{1-|z|}, \qquad z\in \D.
\end{equation}
\end{lemma}
\noindent
A neat and short proof of this fact is provided by Lemma 2 in \cite{pavlovic1999dyakonov}. With this lemma at hand, we now verify that functions in $\A_w$ are Cauchy dual-element of growth spaces $\G_{w^p}$ for a certain range of $p>0$.

\begin{lemma} \thlabel{Awdual}
Let $w$ be a modulus of continuity satisfying condition $(A_2)$ and $0< p < 1-\alpha$. Then we have the containment $\A_w \subseteq \mathcal{F}_{w^p}$. 
\end{lemma}
\begin{proof}
Pick an arbitrary function $f\in \A_{w}$. According to \thref{AwinBw}, we have 
\[
\abs{f'(z)} \leq C \norm{f}_{\A_w} \frac{w(1-|z|)}{1-|z|}, \qquad z\in \D.
\]
Using this we get
\[
\int_{\D} \abs{f'(z)} \frac{dA(z)}{w(1-|z|)^p} \leq C \norm{f}_{\A_w} \int_{\D}  \frac{w^{1-p}(1-|z|)}{1-|z|} dA(z) \leq C' \norm{f}_{\A_w} \int_{0}^1 \frac{w^{1-p}(t)}{t} dt.
\]
The assumption $\alpha <1-p$ and condition $(A_2)$ ensure that the integral is finite, hence we conclude that $P_+$ maps $\A_{w} \subseteq \mathcal{F}_{w^p}$.

\end{proof}

As we shall see, the proof of \thref{THM:Model} will require a deep result on inner factors of functions in $\A_w$ due to N.A. Shirokov in \cite{shirokov1982zero}, which will allow us to construct $\A_w$-functions in model spaces $K_{S_{\mu}}$ for singular measures $\mu$ supported on sets of finite $w$-entropy. Shirokov's Theorem asserts that for any positive finite singular Borel measure $\mu$ supported in a set $E\subset \T$ of finite $w$-entropy, there exists an outer function $f \in \A_w$ such that $fS_{\mu} \in \A_w$. For our purposes, we shall however need the following simple 
modification which we state for future references.

\begin{thm}[Shirokov] \thlabel{THM:Shirokov} Let $w$ be a modulus of continuity and $E\subset \T$ be a set of finite $w$-entropy. Then for any positive finite singular Borel measure $\mu$ supported on $E$, there exists an outer function $f$ in $\A_w$ such that the product $\conj{f}S_\mu$ belongs to $C_w(\T)$.
\end{thm}
\begin{proof}
Let $f$ be the outer function in $\A_w$ that satisfies $fS_\mu \in \A_w$ from Shirokov's Theorem. Now note that if $\zeta, \xi \in \T$, then by triangle inequality
\begin{equation*}
\abs{\conj{f(\zeta)}S_\mu (\zeta) -  \conj{f(\xi)}S_\mu (\xi) } = \abs{ f(\zeta)S_\mu(\xi) - f(\xi)S_\mu(\zeta)} \leq 
2\abs{f(\zeta)-f(\xi)} + \abs{fS_\mu(\xi) - fS_\mu(\zeta) }. 
\end{equation*}
The claim now readily follows from the fact that both $f$ and $fS_\mu$ belong to $\A_w$. 
\end{proof}



\subsection{Approximation in model spaces}

\begin{proof}[Proof of \thref{THM:Model}]
The proof will be divided in different steps.
\proofpart{1}{Existence of $\A_w$ functions:}
We first prove that if $\mu$ does not charge any set of finite $w$-entropy, then $\A_w \cap K_{S_{\mu}} = \{0\}$. Observe that $E$ has finite $w$-entropy if and only if it is finite $w^p$-entropy for any $p>0$, hence by \thref{Thm:GenKR} $S_{\mu}$ is cyclic in $\G_{w^p}$ for any $p>0$. 
Now by \thref{cyclicmodel} this is equivalent to 
\[
(\G_{w^p})' \cap K_{S_\mu} =\{0\}, \qquad p>0.
\]
Choosing $0<p<1-\alpha$ where $0<\alpha <1$ is as in condition $(A_2)$, it now follows from the containment $\A_w \subseteq \mathcal{F}_{w^p}$ in \thref{Awdual} that $\A_w \cap K_{S_\mu}=\{0\}$ as well.

Conversely, suppose that $\mu$ charges a $w$-set $E\subset \T$ and set $\nu := \mu \lvert_E$. We first show that there exists a non-trivial function $f\in \A_w \cap K_{S_\nu}$. That done, part $(i)$ of \thref{THM:Model} is established. Again, observe that since $E$ is a $w$-set is also a $w^p$-set for any $p>0$ and recall that from assumption $(A_2)$ there exists a number $0<\alpha<1$ such that $w^{\alpha +1}$ is a modulus of continuity and $w^{\alpha}$ satisfies the Dini-condition. According to \thref{THM:Shirokov} there exists an outer function $f:\D \to \D$ in $\A_{w^{\alpha +1}}$ such that $\conj{f}S_{\mu} \in C_{w^{\alpha +1}}(\T)$. In particular, if $k_{\lambda}(\zeta) := 1/(1-\conj{\zeta}\lambda)$ denotes the Cauchy kernels with $\lambda \in \D$, we have that each function $g_\lambda :=\conj{f}S_{\nu}k_{\lambda}$ belongs to $C_{w^{\alpha +1}}(\T)$. According to \thref{P+reg}, it follows that $P_+(g_\lambda) \in \A_w$, for each $\lambda \in \D$. Now since Toeplitz operators with co-analytic symbol $f$ preserve smooth functions on $\T$, we have that $T_{\conj{f}}(k_\lambda)$ is also smooth on $\T$, hence in particular belongs to $\A_w$. Now if $\kappa_{S_\nu}(z, \lambda) = \frac{1-\conj{S_\nu(\lambda)}S_{\nu}(z)}{1-\conj{\lambda} z}$ with $z,\lambda \in \D$ denotes the reproducing kernels of the model space $K_{S_\nu}$, we conclude that for any $\lambda \in \D$,
\[
T_{\conj{f}} (\kappa_{S_\nu}(\cdot, \lambda) ) (z) = T_{\conj{f}}(k_{\lambda})(z)  - \conj{S_\nu(\lambda)} P_+(g_\lambda )(z)
\]
is function belonging to $\A_w \cap K_{S_\nu}$. This establishes that $\mu$ does not charge any $w$-set if and only if $\A_w \cap K_{S_\mu}= \{0\}$.

\proofpart{2}{Density of $\A_w$ functions:}
We now turn to proof of the statement about density. To this end, we first establish the claim that the linear span of $\{T_{\conj{f}} (\kappa_{S_\nu}(\cdot, \lambda) )\}_{\lambda \in \D}$ defined in the previous paragraphs is in fact a dense subset of $K_{S_\nu}$. To see this, let $g\in K_{S_\mu}$ be an annihilator of set, that is 
\[
\int_{\T} f(\zeta)g(\zeta) \conj{ \kappa_{S_\nu}(\zeta, \lambda)} dm(\zeta) = 0, \qquad \forall \lambda \in \D.
\]
But this implies that $fg \in (K_{S_\nu})^\perp = S_{\nu}H^2$, and since $f$ is outer, we conclude that the function $g \in K_{S_\nu} \cap S_\nu H^2 = \{0\}$. This proves that $\A_w \cap K_{S_\nu}$ is dense in $K_{S_\nu}$. In fact, since any finite Blaschke product $B_0$ extends analytically across $\T$, the same argument as before actually shows that 
$\A_w \cap K_{B_0 S_\nu}$ is dense in $K_{B_0 S_\nu}$
for any finite Blaschke product $B_0$ and singular measure $\nu$ supported on a $w$-set. 

\proofpart{3}{$\A_w \cap K_{BS_{\mu_P}}$ dense in $K_{BS_{\mu_P}}$:}

To pass to the general claim, fix an arbitrary Blashcke product $B$ and let $\{B_N\}_N$ denote a sequence of finite truncations of $B$ which converge to $B$ uniformly on compact subsets of $\D$. Recall that by definition of $\mu_P$, there exists an increasing union of sets $\{E_N\}_N$ each having finite $w$-entropy (finite unions of sets with finite $w$-entropy again have finite $w$-entropy), such that $\nu_N := \mu_{E_N}$ increase up to $\mu_P$ in total variation norm. We now form the sequence of inner function $\{\Theta_N\}_N$ defined by $\Theta_N = B_N S_{\nu_N}$, and easily verify that $\Theta_N$ converges to $BS_{\mu_P}$ pointwise on $\D$ and that each $\Theta_N$ is by construction a divisor of $BS_{\mu_P}$. By the development in Step 2, $\A_w \cap K_{\Theta_N}$ is dense in $K_{\Theta_N}$ for each $N$, hence invoking \thref{absdens} we conclude that $\A_w \cap K_{BS_{\mu_P}}$ dense in $K_{BS_{\mu_P}}$, which establishes one direction of the statement about density. 

\proofpart{4}{Showing that the closure of $\A_w \cap K_{\Theta}$ equals $K_{BS_{\mu_P}}$:} It remains to deduce the statement phrased in the heading of this step. Write $\Theta= \Theta_P \Theta_C$ where $\Theta_P := BS_{\mu_P}$ and $\Theta_C := S_{\mu_C}$ and recall that a simple computation involving reproducing kernels gives the following orthogonal decomposition in $K_\Theta$ 
\[
K_{\Theta}= K_{\Theta_P} \oplus \Theta_P K_{\Theta_C}.
\]
Now pick an arbitrary element $g \in H^\infty \cap K_{\Theta_C}$ (for instance, any reproducing kernel in $K_{\Theta_C}$ will do) and note that by $(i)$ of \thref{Thm:GenKR}, we can find a sequence of analytic polynomials $\{Q_n\}_n$ such that $\Theta_C Q_n $ converges to $g$ in $\G_{w^p}$ for any $p>0$. Multiplying by $\Theta_P$, we conclude that $\Theta Q_n$ converges to $\Theta_P g$ in $\G_{w^p}$. Now using \thref{Awdual}, we get that $\A_w \subset (\G_{w^p})'$ for $0<p<1-\alpha$, and hence for any $f\in \A_w \cap K_{\Theta}$ we have
\[
\int_{\T} \Theta_P(\zeta) g(\zeta) \conj{f(\zeta)} dm(\zeta) = \lim_{n} \int_{\T} \Theta(\zeta) Q_n(\zeta) \conj{f(\zeta)} dm(\zeta) = 0.
\]
This argument shows that for any $g \in H^\infty \cap K_{\Theta_C}$ the product $\Theta_P g$ is contained in $\left( \A_w \cap K_{\Theta} \right)^{\perp}$. Taking closures and using the fact that reproducing kernels are dense, we obtain the containment $(K_{\Theta_P})^{\perp} = \Theta_P K_{\Theta_C} \subseteq \left( \A_w \cap K_{\Theta} \right)^{\perp}$, which implies that the closure of $\A_w \cap K_{\Theta}$ in $K_\Theta$ is contained in $K_{\Theta_P}$. Using this in conjunction with the previously established claim that $\A_w \cap K_{\Theta_P} \subseteq \A_w \cap K_{\Theta}$ is dense in $K_{\Theta_P}$, we actually conclude that the closure of $\A_w \cap K_\Theta$ equals $K_{\Theta_P}$. This completes the proof of the theorem.

\end{proof}

\bibliographystyle{siam}
\bibliography{mybib}


\end{document}